\documentclass[a4paper,12pt]{amsart}

\usepackage{cellspace}
\usepackage[marginparsep=0.6cm,marginparwidth=3cm]{geometry}

\usepackage{amsmath,amsthm,amsfonts,amssymb,mathtools,mathrsfs,bm}
\usepackage[utf8]{inputenc}

\usepackage[T1]{fontenc}
\usepackage{lmodern}

\usepackage{accents}
\usepackage{nicefrac}
\usepackage{diagbox}
\usepackage{makecell}

\usepackage{tikz-cd}
\usetikzlibrary{babel}
\usetikzlibrary{decorations.pathmorphing}
\usepackage[shortlabels]{enumitem}
\usepackage{float}
\usepackage{subcaption}
\usepackage{multirow,boldline}
\usepackage{graphicx}

\usepackage{xcolor}
\definecolor{colordelink}{rgb}{0,0,0.50}
\definecolor{colordecite}{rgb}{0,0.5,0}
\definecolor{colordeurl}{rgb}{0,0.41,0.5}
\usepackage[pagebackref=true]{hyperref}
\hypersetup{
    colorlinks=true,
    linkcolor=colordelink,
    citecolor=colordecite,
		urlcolor=colordeurl,
}
\usepackage[capitalize,nameinlink,noabbrev]{cleveref}
\usepackage[all]{hypcap}
\usepackage{todonotes}

\usepackage{adjustbox}

\def\im{\operatorname{im}}
\def\dim{\operatorname{dim}}

\def\rank{\operatorname{rank}}
\def\rk{\operatorname{rank}}

\def\codim{\operatorname{codim}}

\def\id{\operatorname{id}}

\def\sgn{\operatorname{sgn}}

\newcommand{\num}{\texttt{\#}}
\newcommand{\Alt}{\textnormal{Alt}}
\newcommand{\alt}{\textnormal{Alt}}

\newcommand{\RR}{\mathbb{R}}

\newcommand{\CC}{\mathbb{C}}
\newcommand{\QQ}{\mathbb{Q}}
\newcommand{\ZZ}{\mathbb{Z}}
\newcommand{\R}{\mathbb{R}}

\newcommand{\eqA}{\mathscr{A}}

\newcommand{\Ascr}{\mathscr{A}}

\newcommand{\Kscr}{\mathscr{K}}

\newcommand{\Fbb}{\mathbb{F}}

\newcommand{\Rbb}{\mathbb{R}}

\newcommand{\Vcal}{\mathcal{V}}
\newcommand{\Wcal}{\mathcal{W}}

\newcommand{\ubm}{\bm{u}}

\makeatletter
\newcommand{\tpitchfork}{%
  \vbox{
    \baselineskip\z@skip
    \lineskip-.52ex
    \lineskiplimit\maxdimen
    \m@th
    \ialign{##\crcr\hidewidth\smash{$-$}\hidewidth\crcr$\pitchfork$\crcr}
  }%
}
\makeatother

\theoremstyle{plain}
\newtheorem{theorem}{Theorem}[section]
\newtheorem{lemma}[theorem]{Lemma}
\newtheorem{corollary}[theorem]{Corollary}
\newtheorem{proposition}[theorem]{Proposition}

\newtheorem*{theorem**}{Theorem\theoremnum}
\newenvironment{theorem*}[1][]{%
  \edef\theoremnum{\if\relax\detokenize{#1}\relax\else~#1\fi}
  \begin{theorem**}
}{%
  \end{theorem**}
} 
\newtheorem*{conjecture**}{Conjecture\theoremnum}
\newenvironment{conjecture*}[1][]{%
  \edef\theoremnum{\if\relax\detokenize{#1}\relax\else~#1\fi}
  \begin{conjecture**}
}{%
  \end{conjecture**}
}

\theoremstyle{definition}
\newtheorem{definition}[theorem]{Definition}
\newtheorem{proposition/definition}[theorem]{Proposition/Definition}
\newtheorem{conjecture}[theorem]{Conjecture}
\newtheorem{example}[theorem]{Example}

\newtheorem{remark}[theorem]{Remark}
\newtheorem{notation}[theorem]{Notation}
\Crefname{notation}{Notation}{Notations}

\theoremstyle{remark}
\newtheorem*{note}{Note}

\begin{document}
\author{I. Breva Ribes, R. Giménez Conejero}
\title[Good real images of complex maps]{Good real images of complex maps}

\address{Departament de Matemátiques,
Universitat de Val\`encia, Campus de Burjassot, 46100 Burjassot
Spain
\newline
Mid Sweden University, Sidsjövägen 2, 852 33 Sundsvall,
Sweden
}
\email{ignacio.breva@uv.es}
\email{Roberto.Gimenez@uv.es}


\begin{abstract}
    We prove several results regarding the homology and homotopy type of images of real maps and their complexification. In particular, we study the local behavior of singular points after deformations. In this context, we prove a restrictive necessary condition for a real perturbation to have the same homology than its complexification, which is known as good real perturbation. We prove the conjecture of Marar and Mond stating that for singularities from $\mathbb{C}^n$ to $\mathbb{C}^{n+1}$, a good real perturbation is homotopy equivalent to its complexification, and show a generalization in other dimensions. Applications to $M$-deformations and other concepts as well as examples are given.
\end{abstract}

\maketitle

\section{Introduction}

We study the relation between the image of a real analytic map and that of its complexification. More precisely, we study this relation for singular unstable maps and their deformations.  Our work is local in nature, so we work with germs and their perturbations, for which a general recommended modern reference is \cite{Mond2020}.

Broadly speaking, we say that a complex map germ has a \textsl{good real perturbation} if the changes in homology by a perturbation can be observed in the real image (this is made precise in \cref{def:grp}). This is similar to the concept of \textsl{good complexifications} (sometimes also \textsl{minimal complexifications}) studied initially by Whitney and Bruhat \cite{Whitney1959} and later by Kulkarni \cite{Kulkarni1978} and Totaro \cite{Totaro2003} among others (cf. \cite{Biswas2016,Biswas2015}). A closed manifold $M$ has a \textsl{good complexification} if it is homotopy equivalent to its \textsl{complexification}\footnote{More precisely, a \textit{good complexification} of a closed manifold $M$ is a smooth affine algebraic variety $V^\RR$ such that $M$ is diffeomorphic to $V^\RR$ (see \cite{Tognoli1973}) and the inclusion of $V^\RR$ into its complexification $V^\CC$ is a homotopy equivalence.}.

Indeed, the similarities go further. In Mond's seminal article on this topic \cite{Mond1996} it was asked, and later conjectured in \cite{Cooper1998} by Cooper and Mond, the following:

\begin{conjecture*}[\ref*{conj:grp}]
If $f:(\CC^n,0)\to(\CC^{n+1},0)$ admits a good real perturbation $f_s^\RR$, then the image of $f_s^\RR$ and of its complexification $f_s$ are homotopy equivalent.
\end{conjecture*}

One of our main results is the proof of this conjecture in \cref{sec:homotopy}. We also prove a generalization in other dimensions, but only in \textsl{corank one} (i.e., when the rank of the differential is the maximum minus one).

In general, the homology of the image a finite map $F$ is given by the multiple point spaces $D^k(F)$, by a spectral sequence that we introduce in \cref{sec: ICSS}. If the map germ $f$ as above has corank one, it turns out that $D^k(f_s)$ has the homotopy type of a wedge of $\mu\big(D^k(f)\big)$ spheres of dimension $\dim D^k(f)$. Our other main theorem is the following.

\begin{theorem*}[\ref*{main}]
If $f$ has corank one and admits a good real perturbation, then $\mu\big(D^k(f)\big)$ is either $0$ or $1$ provided $\dim D^k(f)>0$.
\end{theorem*}

In practice, the methods we use also show whether a particular germ has a good real perturbation or not. Moreover, from all the references that we mention in the following paragraphs, only corank one germs have good real perturbations. Mond and Wik Atique showed that the simplest example of $\Ascr_e$-codimension one germ $(\CC^n,0)\to(\CC^{n+1},0)$ that has corank two (which only exist if $n\geq 5$) does not have a good real perturbation \cite{Mond2003}. No corank two germ is known to have good real perturbations, so it is reasonable to think that our result is actually a characterization of admitting good real perturbations.
\newline

The topic of good real perturbations was initiated by an observation of Goryunov in \cite{Goryunov1991} that motivated the work of Mond \cite{Mond1996}, which we use here. For germs $(\CC^2,0)\to(\CC^3,0)$ a complete classification of germs with good real perturbations was given by Marar and Mond in \cite{Marar1996} (with some interesting visualizations). In \cite{Houston2005}, Houston showed a proof of the conjecture in these dimensions that uses a lemma of \cite{Houston1997} that is false, although the mistake is probably fixable in his argument, we discuss this in \cref{sec:houston}. In that work, Houston gives partial results on the conjecture for germs $(\CC^{2n},0)\to(\CC^{3n},0)$ and gives a family of examples with good real perturbations that also satisfy the conjecture (a generalization for all $n$ of the family $H_k$ for $n=1$). It is also known that all corank one $\Ascr_e$-codimension one germs $(\CC^n,0)\to(\CC^p,0)$ have good real perturbations; for $n+1\geq p$ by Cooper, Mond and Wik Atique in \cite{Cooper2002} and by Houston in \cite{Houston2005} (cf. \cite{Houston2002a,Cooper1993,Houston2005b}). Finally, Cooper and Mond showed some relation of the monodromy of certain good real pictures and their complexification in \cite{Cooper1998}. Independently, McCrory and Parusi\'nski proved stronger results with the same flavour in \cite{McCrory1997}.

However, before these articles, Guse\u{i}n-Zade \cite{GuseinZade1974} and also A'Campo \cite{ACampo1975} proved the celebrated fact that there always exist a real (stable) perturbation of a plane curve with $\delta$ real nodes, which shows that curve singularities always have a good real perturbation (cf. \cref{fig:curvasE6} below). This was generalized to germs $(\CC^n,0)\to(\CC^{2n},0)$ that are $\Ascr$\textsl{-simple} by Klotz, Pop and Rieger in \cite{Klotz2007}. Actually, these two works are focused on \textsl{$M$-deformations} rather than on good real perturbations, but it is easy to see (in particular, from our work here) that these two concepts are equivalent in those dimensions. An \textit{$M$-perturbation} is the analogous concept of an $M$-morsification of a function germ: a real stable perturbation that has all the 0-dimensional singularities of the complex perturbation ($M$ as in \textsl{maximal}). $M$-perturbations also exist for all $\Ascr$\textsl{-simple} germs $f:(\CC^n,0)\to(\CC^p,0)$ of corank one such that $n\geq p$ by Rieger and Ruas \cite{Rieger2005} and $p=n+1$ (but $n\neq4$) by Rieger, Ruas and Wik Atique \cite{Rieger2008a} (cf. \cite{Rieger2008}). We show in \cref{sec:otherdeformations} that any good real perturbation in corank one is also an $M$-perturbation, for $n<p$.


%
%



\subsection{Structure of the article}

In \cref{sec:preliminaries} basic notation is established about alternating homology, multiple point spaces and some previously known results about good real perturbations.
The Image-Computing Spectral Sequence (ICSS), which is used throughout the paper, and some key results are also presented 
 here.

\cref{sec:whiskers} is dedicated to studying whiskers of map germs, i.e., real points in the target which are the image of non-real points in the source.

The core of the paper lies in \cref{sec: Pruebas}
, where the central result \cref{main} is obtained.

The general proof of \cref{conj:grp} is given in \cref{sec:homotopy}, along with the above mentioned counterexample to Houston's lemma.
Moreover, we also give a proof for the case of $p>n+1$ when the map germ has corank 1.

The last sections are dedicated to several applications of the central results: excellent real perturbations, $M$-deformations and a classification from $\CC^3$ to $\CC^4$ of good real perturbations.
\newline

\textbf{Acknowledgements:} The authors are thankful to Marco Marengon and Diego González-Sánchez for their helpful conversations regarding technical lemmas of this work and to Raúl Oset Sinha for the encouragement and support. 
Both authors were partially supported by Grant PID2021-124577NB-I00 funded by MCIN/AEI/ 10.13039/501100011033 and by “ERDF A way of making Europe”. The first-named author was partially supported by grant UV-INV-PREDOC22-2187086, funded by Universitat de València.
The second-named author was partially supported by NKFIH Grant “Élvonal (Frontier)” KKP 144148. 

\section{Preliminaries}\label{sec:preliminaries}

\subsection{The case of ICIS}\label{sec:icis}

It is natural to ask if we could have a \textit{good real picture} of a hypersurface singularity or, more generally, an isolated complete intersection singularity (ICIS). More precisely, for a complex ICIS $(X,0)\subset(\CC^N,0)$ of dimension $d$, under what circumstances there is an isomorphic ICIS $(X^\CC,0)$ with real equations such that its Milnor fiber $F^\CC$ and $F^\RR\coloneqq F^\CC\cap\RR^N$ (i.e., the \textsl{real Milnor fiber}) have the same homology in dimension $d$. This was answered by Mond. We include its proof since it is what motivated this work.

\begin{proposition}[cf. {\cite[Remark 2-4]{Mond1996}}]\label{prop:iciscase}
Any ICIS $(X^\CC,0)$ of dimension $d>0$ that has a good real picture as described above has Milnor number $0$ or $1$.
\end{proposition}
\begin{proof}
We use the notation above. By Smith Theory (see \cref{{thm:ST_ineq}} below) and  Poincaré-Lefschetz duality, if $X^\CC$ is singular,
\[\begin{aligned}
1+\mu(X^\CC)=\sum_i \beta_i(F^\CC;\Fbb_2)\geq &\beta_i(F^\RR;\Fbb_2) \\
\geq&\beta_0(F^\RR)+\beta_d(F^\RR)\\
\geq&2\beta_d(F^\RR).
\end{aligned}\]
If the ICIS has a good real picture then $\mu(X^\CC)=\beta_d(F^\RR)$, which implies $\mu(X^\CC)=1$.
\end{proof}

It is possible to complete this to a homotopy result. The idea to use Morse theory in the way we do it in the proof of the theorem below was taken from \cite[Lemma 2-2]{Mond1996}. In fact, that proof has a small missing step. Following the notation there, the singularities of the equation $g_t$ may not be Morse singularities, but they are after taking a convenient Morsification and gluing the pieces of the fibers around each singular point of $g_t$. We avoid that step entirely since we already have a unique Morse singularity, as we show now.

\begin{lemma}\label{lem:milnorone}
Any ICIS $(X,0)$ of dimension $d>0$ with Milnor number one is isomorphic to an $A_1$ hypersurface singularity, in the sense that the local algebras are isomorphic.
\end{lemma}
\begin{proof}
ICIS with Milnor number one and positive dimension also have Tjurina number one, by \cite[Theorem]{Looijenga1985}. This shows that they are simple, because any $\Kscr_e$-codimension one singularity (i.e., with Tjurina number one) has finitely-many adjacent singularities ($\Kscr$-equivalence classes of singularities in any possible perturbation), in this case only the smooth one. By the classification of Giusti of simple isolated singularities of complete intersections \cite[Théorème]{Giusti1983}, the ICIS can only be isomorphic to the $A_1$ hypersurface singularity.
\end{proof}

\begin{theorem}\label{thm:iciscase}
If an ICIS $(X^\CC,0)$ has a good real picture, then $F^\RR$ is a deformation retract of $F^\CC$.
\end{theorem}
\begin{proof}
We prove that there is a homotopy equivalence between $F^\RR$ and $F^\CC$ , which is enough by, for example, \cite[Corollary 0.20]{Hatcher2002} (see also \cite[Proposition 0.16]{Hatcher2002}).
\newline

The zero-dimensional case is clear, so we assume that the dimension of $X^\CC$ is positive. After \cref{lem:milnorone}, both $F^\RR$ and $F^\CC$ have codimension one within some smooth hypersurface in their respective ambient spaces $\RR^N$ and $\CC^N$. Therefore, we can assume that they are already hypersurfaces in $\RR^N$ and $\CC^N$ and we can give a Morse-theoretical argument with their unique equation. We also know from \cref{lem:milnorone} that $(X^\CC,0)$ is a Morse singularity, so both $F^\CC$ and $F^\RR$ must be homotopy equivalent to a sphere of dimension $d=N-1$. 

Up to homotopy, one obtains a topological ball in $\CC^N$ by attaching to $F^\CC$ an $N$-cell corresponding to the singularity $(X,0)$. We have a similar argument for the real case and $F^\RR$, so the $N$-cell that we glue in the complex case may be taken to be the same as in the real case. The argument follows the same steps as \cite[Lemma 2-2]{Mond1996}. Let us denote the cell as $E$. Since
\[H^*(F^\CC\cup E, F^\RR\cup E)\cong 0,\]
by excision we also have that
\[H^*(F^\CC, F^\RR)\cong 0.\]
This shows that the inclusion is an isomorphism in all homology groups, which must therefore be a homotopy equivalence. Indeed, if $d>1$, $F^\CC$ and $F^\RR$ are simply connected, so the homotopy equivalence is given by Whitehead's theorem (e.g., Theorem 4.5 in \cite{Hatcher2002}).
 If $d=1$, the argument is trivial.
\end{proof}

\subsection{Multiple points and the Image-Computing Spectral Sequence}\label{sec: ICSS}

One can study images of finite maps by using the ICSS (Image-Computing Spectral Sequence), which is a spectral sequence that has as \textsl{input} the \textsl{alternating homology} of the \textsl{multiple point spaces} of the map and converges to the homology of the image of the map. Let us start by giving the ICSS we use. One of the most general versions is given in \cite[Theorem 5.4]{Houston2007}, cf. \cite{Goryunov1993,Goryunov1995,CisnerosMolina2022,Mond2020}. For us, it suffices the following version.

\begin{theorem}\label{thm: general icss}
Let $f:X\to Y$ be a finite and surjective analytic map between compact subanalitic spaces. Then, there exists a spectral sequence
$$ E^1_{p,q}(f;G)\coloneqq AH_q\big(D^{p+1}(f);G\big)\Longrightarrow H_*\big(f(X);G\big), $$ 
where $G$ is a coefficient group and the differential $d_1$ is induced by the projections $\pi:D^k(f)\to D^{k-1}(f)$ for any $k$.
\end{theorem}

For a general treatment of the multiple point spaces one can see \cite{Houston2007,Houston1999} and \cite[Chapter 2]{Robertothesis}. The following definition is enough for us.

\begin{definition} \label{def: mult spaces f}
The \emph{$k$th-multiple point space} of a finite map or germ $f$, denoted as $D^k(f)$, is defined as follows ($\Fbb=\CC$ or $\RR$):
\begin{itemize}
\item
 If $f\colon X\rightarrow Y$ is a locally stable map between analytic manifolds, then
\[D^k(f)\coloneqq\overline{\big\{    \big(x^{(1)},\dots,x^{(k)}\big)\in X^k   :   f\left(x^{(i)}\right)=f\left(x^{(j)}\right), x^{(i)}\neq x^{(j)}  \big\}}.\]

\item If $f\colon (\Fbb^n,S)\rightarrow(\Fbb^p,0)$ is a stable germ, $D^k(f)$ is the analogous germ in $\big((\Fbb^n)^k,S^k\big)$.

\item
If $f\colon(\Fbb^n,S)\rightarrow(\Fbb^p,0)$ is finite it has a stable unfolding $F(x,\ubm)=\big(f_{\ubm}(x),\ubm\big)$, see \cite[Proposition 7.2]{Mond2020}. Then, $D^k(f)$ is the analytic space germ in $\big((\Fbb^n)^k,S^k\big)$ given by 
\[ D^k(f)\coloneqq D^k(F)\cap\left\{\ubm=0\right\}.\]
\end{itemize}
\end{definition}

There is an obvious action of the group of permutations $\Sigma_k$ in $D^k(f)$ by permuting copies of $X^k$. Hence, we can consider the induced action of $\Sigma_k$ at the (simplicial) chain level and take the alternating isotype of the chain complex of $D^k(f)$, 
\[\begin{aligned}
C^\Alt_*\big(D^k(f);G\big)=& C^\Alt_*\big(D^k(f);\ZZ\big)\otimes G, \text{where}\\
C^\Alt_n\big(D^k(f);\ZZ\big)=& \left\{c\in C_n\big(D^k(f)\big):\sigma^*(c)=\sgn(\sigma)c\ \forall\sigma\in\Sigma_k\right\}.
\end{aligned}\]
It is easy to see that this is a chain complex with the restriction of the boundary. Then, $AH_*$ is the homology of such chain complex. 
There is another \textsl{alternating} object one can consider, one can take the induced action in homology and then take the alternating isotype:
\[ H^\Alt_n \big(D^k(f); \ZZ\big)=\left\{\omega\in H_n\big(D^k(f)\big) : \sigma ^* \omega = \sgn(\sigma) \omega\right\}.\]

\begin{remark}\label{rem:HaltToAH}
These two objects are not isomorphic in general, see \cite[Example 5.1]{Houston1999}. However, when we take $AH$, or $H^\Alt$, on the fiber of a $\Sigma_k$-invariant ICIS they are isomorphic by \cite[Theorem 2.1.2]{Goryunov1995}. This is our situation for $D^k(f)$, as we mention in \cref{lem: dkgamma} below. Furthermore, they are also isomorphic when the coefficient group is a field by \cite[Proposition 10.1]{Mond2020}. We use these two cases of isomorphisms and compute $H^\Alt_*$ instead of $AH_*$, since it is easier.
\end{remark}

\subsection{Complex deformations}\label{sec:complexdef}

Let us go back to the $D^k(f)$ spaces. For $\Ascr$-finite monogerms $f:(\CC^n,0)\to(\CC^p,0)$, $n<p$, the celebrated theorem of Houston \cite[Theorem 4.6]{Houston1997} shows that the stable perturbation $f_s$ is such that $AH_i\big(D^k(f)\big)=0$ if $i\neq \dim_\CC D^k(f_s)$ and free abelian for $i = \dim_\CC D^k(f_s)$. However, in general, $H_*\big(D^k(f_s)\big)$ is non-trivial in many dimensions (see \cite{Mond2016}). This situation is radically simpler if $f$ has \textit{corank one} (i.e., $\rk df_0=n-1$) as \cref{lem: dkgamma} below, known as the \textit{Marar-Mond criterion}, shows.

\begin{definition}\label{def:dksigma}
For any element $\sigma\in\Sigma_k$, we define $D^k(f)^\sigma$ as (the isomorphism type of) the subspace of $D^k(f)$ given by the fixed points of $\sigma$. 
\end{definition}

The relevance of these spaces is explained in \cref{sec: Pruebas} below.

A simple point-wise computation shows that $D^k(f)^{\tau\sigma\tau^{-1}}=\tau D^k(f)^\sigma$. Since conjugacy classes of $\Sigma_k$ are identified with partitions of $k$, so are the isomorphism type of the spaces $D^k(f)^\sigma$. E.g., $\sigma=(1\ 2\ 6\ 3)(4\ 5)(8\ 9)(7)(10)$ corresponds to the partition $\gamma(k)=(r_1,\dots, r_m)=(4,2,2,1,1)$. It will be useful to give the notation $\sigma^\num$ to the \textit{number of cycles} in a given permutation $\sigma$. In the previous example, $\sigma^{\num}=5$.

\begin{lemma}[See {\cite[Corollary 2.15]{Marar1989}} and {\cite[Corollary 2.8]{Houston2010}}]\label{lem: dkgamma}
If $f:(\CC^n,S)\rightarrow (\CC^p,0)$ is a finite germ of corank $1$, $n < p$, and $\sigma\in\Sigma_k$, then:
\begin{enumerate}
	\item If $f$ is stable, $D^k(f)^\sigma$ is smooth of dimension $p-k(p-n)-k+\sigma^\num$, or empty.\label{iPD}
	\item $\eqA_e-codim(f)$ is finite if, and only if:\label{iiPD}
	\begin{enumerate}
		\item for each $k$ with $p-k(p-n)-k+\sigma^\num\geq0$, $D^k(f)^\sigma$ is empty or an ICIS of dimension $p-k(p-n)-k+\sigma^\num$,
		\item for each $k$ with $p-k(p-n)-k+\sigma^\num<0$, $D^k(f)^\sigma$ is a subset of $S^k$, possibly empty.
	\end{enumerate}
\end{enumerate}
\label{DP}
\end{lemma}

From the previous result we have the following definition.

\begin{definition}\label{def:expecteddimension}
We will say that the \textit{expected dimension} of $D^k(f)$ (and $D^k(f_t)$) is $d_k\coloneqq p-k(p-n)$, and the \textit{expected dimension} of $D^k(f)^\sigma$ (and $D^k(f_t)^\sigma$) is $d_k^\sigma\coloneqq p-k(p-n)-k+\sigma^\num$. If we want to emphasize the map, we shall write $d_k(f)$ instead of $d_k$.
\end{definition}

\begin{remark}\label{rem: eqs Dk}
For (mono-) germs of corank one, $D^k(f)$ can be seen as a subset of $\CC^{n+k-1}$ (see, for example, \cite[end of p. 371]{Mond1987} or \cite[p. 52]{GimenezConejero2022}). In that case, it is possible to give equations in an easy way by taking \textsl{divided differences}, this is what we will do in \cref{sec:C3C4}.
These equations will not be equivariant, but it is possible to give equivariant equations. See \cite[Section 9.5]{Mond2020}.
\end{remark}

\begin{remark}\label{rem:propertiesDk}
\cref{lem: dkgamma} shows several things. Assume that $f$ is an $\eqA$-finite monogerm, $k$ and $\sigma$ are such that $d_k,d_k^\sigma\geq0$. For a stable perturbation $f_s$ of $f$, $D^k(f_s)$ is the Milnor fiber of the ICIS $D^k(f)$, since it is a smooth deformation. Also, $D^k(f_s)^\sigma$ is empty if, and only if, $D^k(f_s)$ is empty, since $(0;\dots; 0)\in D^k(f)$ otherwise. 
This leads to the following: since any singular ICIS has positive Milnor number, one expects that there is always non-trivial alternating homology in $D^k(f_s)$ if $D^k(f)$ is singular. This was recently shown by Giménez Conejero in \cite[Theorem 4.8]{GimenezConejero2022c}. He also gave a way of computing the alternating homology by using the spaces $D^k(f_s)^\sigma$ that we mention later.
\end{remark}

Finally, our techniques also allows us to address more than just the image of stable deformations $f_s$, we can study the (closure of the) set of points in the image with at least $k$ preimages. These objects were studied for example in \cite{Houston2001,Houston2002,Houston2002a}.

\begin{definition}
Let $f\colon X\rightarrow Y$ be a locally stable mapping between complex manifolds, and $\pi:D^k(f)\to X$ the induced map from the projection onto the first coordinate $X^k\to X$. Then, the \textit{$k$-th multiple points in the image} is the set $M^k(f)\coloneqq f\circ\pi\big(D^k(f)\big)$. In particular, $M^1(f)=\im(f)$.
\end{definition}

\subsection{Real deformations}

Usually, e.g. \cite{Mond1996,Cooper1998,Cooper2002,Houston2002a}, the definition of good real perturbation concerns only the non-trivial reduced Betti numbers of the image (or, more generally, discriminant) of $f_s$, this is the reason we introduce the concept of \textsl{complete} good real picture below. One of the goals of this work is to show relations between the following definitions, see \cref{cor:GRPisCGRP,prop:good=excellent}.

\begin{definition}\label{def:grp}
Let $f:(\CC^n,S)\to(\CC^p,0)$ and $f^\RR:(\RR^n,S')\to(\RR^p,0)$ be two $\Ascr$-finite germs, with respective stable perturbations $f_s$ and $f^\RR_s$. Assume that the complexification $f^\CC$ of $f^\RR$ is $\Ascr$-equivalent to $f$. Then, we say that
\begin{enumerate}[label=(\roman*)]
	\item $f^\RR_s$ is a \textit{good real perturbation} of $f$ if $\beta_i\big(\im(f^\RR_s)\big)=\beta_i\big(\im(f_s)\big)$ whenever $\beta_i\big(\im(f_s)\big)\neq0$. In that case, we say that $f$ has a \textit{good real picture}.
   \item $f^\RR_s$ is a \textit{complete good real perturbation} of $f$ if  the discriminants of $f^\RR_s$ and $f_s$ have the same Betti numbers.
   \item $f^\RR_s$ is an \textit{excellent real perturbation} of $f$ if $\beta_i\big(M^k(f^\RR_s)\big)=\beta_i\big(M^k(f_s)\big)$ for all $i,k$.
\end{enumerate}
\end{definition}

One of the key problems in the search of a good real perturbation is that there is not \textsl{a unique} stable perturbation in the real case (up to homeomorphisms), compared to the complex case which is unique up to homeomorphism. Their topological types depend on the unfolding itself and the sign of the parameters. See, for example, \cref{fig:curvasE6}.

\begin{figure}
	\centering
  \includegraphics[width=1\textwidth]{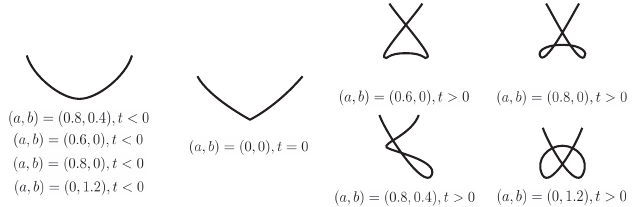}
\caption{Topological types of the deformation $f^\RR_{a,b,t}(x)=(x^3-tx,x^4-tax^2-tbx)$ for different values of $a,b$ and $t$.
}
\label{fig:curvasE6}
\end{figure}

As we were introducing before, we also deal with $M$-perturbations.

\begin{definition}\label{def:Mperturbation}
Let $f:(\CC^n,S)\to(\CC^p,0)$ and $f^\RR:(\RR^n,S')\to(\RR^p,0)$ be two $\Ascr$-finite germs, with respective stable perturbations $f_s$ and $f^\RR_s$. Assume that the complexification $f^\CC$ of $f^\RR$ is $\Ascr$-equivalent to $f$. Then, we say that $f_s^\RR$ is an \textit{$M$-perturbation} of $f$ (also, of $f^\RR$) if it has the same stable singularities of dimension zero than $f_s$.
\end{definition}

\subsection{Equivariant Smith theory}\label{sec:equivariant_st}

A natural way to relate a real variety and its complexification is by the action of $\nicefrac{\ZZ}{2\ZZ}$ by conjugation and Smith Theory, which is well exemplified in the following theorem of Floyd, \cite[Theorem 4.4]{Floyd1952}.

\begin{theorem}\label{thm:ST_ineq}
    Let $X$ be a locally compact finite-dimensional Hausdorff space.
    Let $G$ be a cyclic group of prime order $p$ acting over $X$ and denote by $X^G$ the set of points fixed by the action.
    Then, for each $N \geq 0$,
    \[ \sum_{i=N}^\infty \beta_i(X;\Fbb_p)\geq\sum_{i=N}^\infty \beta_i(X^{G};\Fbb_p).\]
\end{theorem}

Since the image of a finite map is given by the alternating homology of the multiple point spaces, as introduced above, we need a similar version of this theorem for the homology $AH_*(X,\Fbb_p)$ of the alternating chain complex $C_*^\Alt(X, \Fbb_p)$.

\begin{definition}
    Let $X$ be a simplicial complex, we say that it is a simplicial $G$-complex if $G$ acts over $X$ taking simplexes to simplexes and fixing them point by point whenever they are fixed.
\end{definition}

Assume now that $X$ is a simplicial $\Sigma_k$-complex, where $\Sigma_k$ is the permutation group. This is precisely the situation of multiple point spaces of corank one germs, see \cite[Section 5]{Houston2007}.
Let $A\beta_i(X,\Fbb_p)$ denote the rank of $AH_i(X,\Fbb_p)$. We want to show the following.

\begin{theorem}\label{thm:STtau}
Let $G$ be a $p$-group, $\Sigma_k$ be a group of permutations and $X$ a simplicial complex such that $X$ is both a simplicial $G$ and $\Sigma_k$-complex, and assume that both actions commute. Then, for any $N\geq0$,
\[ \sum_{i=N}^\infty A\beta_i(X;\Fbb_p)\geq\sum_{i=N}^\infty A\beta_i(X^{G};\Fbb_p).\]
\end{theorem}
To obtain Floyd's result, he makes use of Smith's special homology groups.
In order to obtain ours, we will follow the approach of Putman in his notes \cite{PutmanSTnotes}, where he uses Bredon coefficient systems to obtain Smith's groups and prove the same result.

This approach of Bredon coefficient systems can be generalized to the alternating homology $AH_*$ case and the desired results follow.
However, for the sake of brevity, we will instead  construct directly the chain complexes and special homology groups needed for the proof.

\begin{remark}
    It is noted by Putman that it is enough to prove \cref{thm:STtau} for the case that $|G| = p$ by an inductive argument: if $|G| = p^{k+1}$, then it admits a proper non-trivial normal subgroup $G'$ and $G/G'$ acts on $X^{G'}$ in a natural way with fixed points
    $$\left(X^{G'}\right)^{G/G'} = X^G.$$
\end{remark}

Let $G$ be cyclic of order $p$ generated by $g$ and $\Fbb_p [G] \coloneqq\frac{\Fbb_p[g]}{(1-g^p)}$.
If $\eta \coloneqq 1 - g$ and $\kappa \coloneqq 1 + g + \cdots + g^{p-1}$ are elements of $\Fbb_p[G]$, then the following lemma is an easy exercise of modular arithmetic (see \cite[Lemmas 4.1 to 4.3]{PutmanSTnotes}).


\begin{lemma}
The following statements hold:
\begin{enumerate}[label=(\roman*)]
    \item $\kappa = \eta^{p-1}$
    \item The kernel of the homomorphism $\Fbb[G]_p\rightarrow \eta\cdot\Fbb_p[G]$ given by multiplication by $\eta$ is 1-dimensional and spanned by $\kappa$.
    \item $\kappa \in \eta^i\cdot\Fbb_p[G]$ for every $i = 0,\ldots, p-1$.
\end{enumerate}
\end{lemma}

As a consequence:

\begin{lemma}[see {\cite[Lemma 4.4]{PutmanSTnotes}}]\label{lem:fp_short_exact_1}
    For every $i=0,\ldots,p-1$, there is a short exact sequence given by the inclusion and multiplication by $\eta$,
    \[0\rightarrow \kappa \cdot \Fbb_p[G] \rightarrow \eta^i \cdot \Fbb_p[G]  \rightarrow \eta^{i+1} \cdot \Fbb_p[G]  \rightarrow 0. \]
\end{lemma}

Using both lemmas we can construct the following short exact sequence. We include its proof for completion, see also \cite[Lemmas 4.5 and 4.6]{PutmanSTnotes}.

\begin{lemma}\label{lem:fp_short_exact_2}
    For any $i=0,\ldots,p$, there is a short exact sequence 
    $$0\rightarrow \bar\rho \cdot \Fbb[G] \rightarrow  \Fbb[G]  \rightarrow \rho \cdot \Fbb[G]  \rightarrow 0,$$
    where $\rho \coloneqq \eta^i$ and $\bar\rho \coloneqq \eta^{p-i}$.
\end{lemma}
\begin{proof}
    The inclusion of $\bar\rho \cdot \Fbb[G]$ in $\Fbb[G]$ is clearly injective, and product by $\rho$ is a surjective map from $\Fbb[G]$ to $\rho\cdot\Fbb[G]$.
    Now, exactness follows by computing the dimension of $\bar\rho \cdot \Fbb_p[G]$ and $\rho \cdot \Fbb_p[G]$ over $\Fbb_p$. Observe that $\dim_{\Fbb_p} \eta^0\cdot\Fbb_p[G] =\dim_{\Fbb_p} \Fbb[G] = p$ and that, by \cref{lem:fp_short_exact_1},
    $$\dim_{\Fbb_p} \eta^{i+1} \cdot \Fbb_p[G] = \dim_{\Fbb_p} \eta^i\cdot\Fbb_p[G] - \dim_{\Fbb_p}\kappa\cdot\Fbb_p[G] = \dim_{\Fbb_p}\eta^i\cdot\Fbb_p[G] - 1.$$
    Hence, $$\dim_{\Fbb_p}\rho \cdot \Fbb_p[G] + \dim_{\Fbb_p}\bar\rho \cdot \Fbb_p[G] = (p-i)+i = p = \dim_{\Fbb_p}\Fbb_p[G].\qedhere$$
\end{proof}

Now, assume that $X$ is a simplicial $G$-complex.
Then the set of fixed points $X^G$ is a simplicial subcomplex.
Since the action is simplicial we can consider the quotient by action of $G$ restricted to the set $X^{(n)}$ consisting of $n$-simplexes of $X$.
Then, we can just write
\begin{align*}
    C_*(X,\Fbb_p) &= \bigoplus_{Y\in X^{(*)}/G} \left\lbrace \sum_{\Delta \in Y}a_\Delta\cdot \Delta \,|\, a_\Delta\in \Fbb_p\right\rbrace, \text{ and}\\
    C_*(X^G,\Fbb_p) &= \bigoplus_{\substack{Y\in X^{(*)}/G \\ |Y| = 1}} \big\lbrace a_\Delta\cdot \Delta \,|\, a_\Delta\in \Fbb_p, \Delta \in Y\big\rbrace.
\end{align*}
Fix an $i=1,\ldots,p-1$.
Then the polynomial $\rho = \eta^i$ acts naturally over $C_*(X,\Fbb_p)$.
Define the image of this action as the complex $ C_*^\rho(X,\Fbb_p)$.

Lastly, assume that $X$ is both a simplicial $G$-complex and a simplicial $\Sigma_k$-complex, with $\Sigma_k$ being the group of permutations.
Assume moreover that the action of $G$ commutes with that of $\Sigma_k$.
Then, it makes sense to write $C_*^{\Alt,\rho}(X,\Fbb_p)$ as the restriction of the action of $\rho$ over the subcomplex $C_*^\Alt(X,\Fbb_p) = \lbrace c\in C_*(X,\ZZ) \,|\, \sigma c = \sgn(\sigma)c\text{ for each }\sigma\in\Sigma_k\rbrace \otimes \Fbb_p$. 

\begin{proposition}\label{prop:smith_alt_se_sequence}
    If $X$ is a simplicial $G$ and $\Sigma_k$-complex whose actions commute, then there is a short exact sequence of complexes
    $$0\rightarrow  C_*^{\Alt,\bar\rho}(X,\Fbb_p)\oplus C_*^\Alt(X^G,\Fbb_p)\rightarrow C_*^\Alt(X,\Fbb_p) \rightarrow C_*^{\Alt,\rho}(X,\Fbb_p) \rightarrow 0.$$
\end{proposition}
\begin{proof}
    First we prove that the sequence
    \begin{equation}\label{eq:smith_se_sequence}
    0\rightarrow C_*^{\bar\rho}(X,\Fbb_p)\oplus C_*(X^G,\Fbb_p)\xrightarrow{i} C_*(X,\Fbb_p) \xrightarrow{\rho\cdot} C_*^\rho(X,\Fbb_p) \rightarrow 0 
    \end{equation}
    is exact, with $i$ being the inclusion and $\rho\cdot$ the action induced by $\rho$. 
    Then, if $i|$ and $\rho\cdot|$ are the restrictions to alternating chain complexes,
    $$\begin{aligned}
    \ker(\rho\cdot|) &= \ker(\rho\cdot)\cap C_*^\Alt(X,\Fbb_p) = \im(i)\cap C_*^\Alt(X,\Fbb_p)\\
    &= C_*^{\Alt,\bar\rho}(X,\Fbb_p)\oplus C_*^\Alt(X^G,\Fbb_p) = \im(i|),
    \end{aligned}$$
    since both actions commute. 
    All $i$, $i|$, $\rho\cdot$ and $\rho\cdot|$ are chain maps, the first two because they are just the inclusion, and the third and fourth because $X$ is a simplicial $G$ and $\Sigma_k$-complex, making the boundary commute with both actions.
    
    Now, for each orbit $Y\in X^{(*)}/G$, either $|Y| = 1$ or $|Y| = p$.
    Therefore,
    \begin{align*}
        C_*(X,\Fbb_p) & \cong \bigoplus_{\substack{Y\in X^{(*)}/G \\ |Y| = p}} \Fbb_p[G] \oplus \bigoplus_{\substack{Y\in X^{(*)}/G \\ |Y| = 1}} \Fbb_p, \\
        C_*(X^G,\Fbb_p) &\cong \bigoplus_{\substack{Y\in X^{(*)}/G \\ |Y| = 1}} \Fbb_p,\text{ and} \\
        \rho \cdot C_*(X,\Fbb_p) &\cong \bigoplus_{\substack{Y\in X^{(*)}/G \\ |Y| = p}} \rho \cdot \Fbb_p[G].
    \end{align*}
    The last equivalence is due to the fact that the action of $\rho$ takes simplexes fixed by $G$ to the zero chain.
    Therefore, the sequence in \cref{eq:smith_se_sequence} is exact if, and only if, the following sequence is also exact:
    \begin{equation*}
      0\rightarrow \hspace{-5pt} \bigoplus_{\substack{Y\in X^{(*)}/G \\ |Y| = p}} \hspace{-5pt} \bar\rho\cdot\Fbb_p[G] \oplus \hspace{-5pt}\bigoplus_{\substack{Y\in X^{(*)}/G \\ |Y| = 1}} \hspace{-5pt} \Fbb_p \xrightarrow{i} \hspace{-5pt} \bigoplus_{\substack{Y\in X^{(*)}/G \\ |Y| = p}} \hspace{-5pt} \Fbb_p[G] \oplus \hspace{-5pt} \bigoplus_{\substack{Y\in X^{(*)}/G \\ |Y| = 1}} \hspace{-5pt} \Fbb_p \xrightarrow{\rho\cdot} \hspace{-5pt} \bigoplus_{\substack{Y\in X^{(*)}/G \\ |Y| = p}} \hspace{-5pt} \rho \cdot \Fbb_p[G]\rightarrow 0
      \end{equation*}
    where $\rho\cdot$ acts as polynomial multiplication over $\Fbb_p[G]$ and as the zero morphism over $\Fbb_p$.
    Since, by \cref{lem:fp_short_exact_2} $\ker(\cdot \rho|_{\Fbb_p[G]}) = \bar\rho\cdot \Fbb_p[G]$, the result follows.
\end{proof}

\begin{proof}[Proof of \cref{thm:STtau}]
    This proof is the same as the one by Putman in his notes \cite{PutmanSTnotes}.
    The short exact sequence from \cref{prop:smith_alt_se_sequence} induces a long exact sequence
    \[\cdots\rightarrow AH_{k+1}^\rho(X) \rightarrow AH_k^{\bar\rho}(X) \oplus AH_k(X^G) \rightarrow AH_k(X) \to AH_k^\rho(X) \rightarrow\cdots\]
    where $AH_*^{\rho}(X)$ is the homology of the chain complex $C_*^{\Alt,\rho}(X,\Fbb_p)$ and similarly for $\bar\rho$.
    All homology groups are taken with coefficients in $\Fbb_p$.
    Notice that a similar long exact sequence holds after interchanging $\rho$ and $\bar\rho$.
    
    Define
    $$a_i \coloneqq \dim (AH_i^{\rho}(X)) \text{ and } \bar a_i \coloneqq \dim (AH_i^{\bar\rho}(X)).$$
    Then, for each $k\geq 0$, using both mentioned long exact sequences,
    \begin{align*}
        \dim AH_k(X^G) &\leq a_{k+1} - \bar a_k + \dim AH_k(X) \text{, and}\\
        \dim AH_k(X^G) &\leq \bar a_{k+1} - a_k +  \dim AH_k(X)
    \end{align*}
    Since $X$ is a finite-dimensional simplicial complex, there is a big enough $N$ such that for $k>N$, $a_k = \bar a_k = 0$. 
    Let $n \geq 0$; if $N-n$ is even,
    \begin{align*}
        \sum_{k=n}^N \dim AH_k(X^G)& \leq \sum_{k=n}^N\dim AH_k(X) + (a_{n+1} - \bar a_n) + 
        \cdots + (\bar a_{N+1} - a_N) \\ & =  \sum_{k=n}^N\dim AH_k(X) - \bar a_{n}.
    \end{align*}
    If $N - n$ is odd, the last term in the telescopic sum would be $a_{N+1} - \bar a_N$.
    In both cases, since $a_n,\bar a_n\geq 0$, the result follows.
\end{proof}

\section{Whiskers of stable map-germs}\label{sec:whiskers}

Let $f^\RR:(\RR^n,S')\to(\RR^p,0)$ be an $\Ascr$-finite germ and $f^\CC$ its complexification, $n<p$. Consider their stable perturbations $f^\RR_s$ and $f^\CC_s$. In this section, we discuss the difference between the sets $\im(f^\RR_s)\subseteq\im(f^\CC_s)\cap\RR^p\subseteq\im(f^\CC_s)$. Observe that $\im(f^\CC_s)\cap\RR^p$ is simply the Zariski closure of $\im(f^\RR_s)$.

\begin{definition}\label{def:whiskers}
Let $g$ be a locally stable holomorphic map between the (possibly disjoint) unions of open balls $U\subseteq \CC^n$ and $V\subseteq \CC^p$, $n<p$. Assume that $g$ is the complexification of a real analytic map $g^\RR$, with the respective domain and codomain by open balls. Then, the \textit{whiskers} of $g$ are
\[\Wcal(g) \coloneqq\overline{\im(g)\cap\RR^p\setminus \im(g^\RR)}.\]
\end{definition}

\begin{notation}\label{notation}
Let $g$ be a map as in \cref{def:whiskers}. We fix the following notation:
\begin{enumerate}[label=(\roman*)]
	\item $W^k(g)$ are the points of the multiple point space $D^k(g)$ that project to $\Wcal(g)$;
   \item $g\rvert_\Wcal$ is the restriction of $g$ to $g^{-1}\big(\Wcal(g)\big)$; and
   \item $g\rvert_{\RR^p}$ is the restriction of $g$ to $g^{-1}(V\cap\RR^p)$.
\end{enumerate}
\end{notation}

The name {\it whiskers} was first used in \cite{Marar1996}, where they study map germs as above in the case $(n,p)=(2,3)$. In these dimensions, whiskers are real analytic sets of real dimension 1, hence the name (see also \cref{fig:whiskers}).

\begin{figure}
\centering
\includegraphics[width=.8\linewidth]{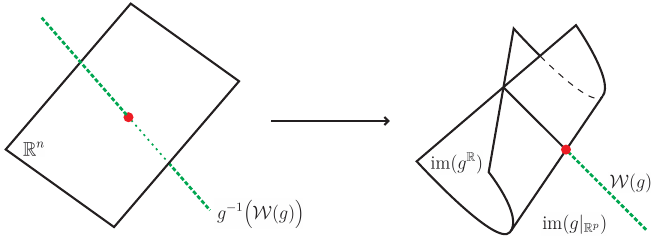}
\caption{Representation of a map $g^\RR$ together with its whiskers $\Wcal(g)$ (green, dashed line) and a non-immersive point (red).}
\label{fig:whiskers}
\end{figure}

\begin{proposition}[cf. {\cite[Lemma 2-1]{Mond1996}}]\label{prop:whisk properties}
In the conditions of \cref{def:whiskers},
\begin{enumerate}[label=(\roman*)]
	\item\label{prop:whisk properties1} $\Wcal(g)$ lies in an analytic subset of complex dimension less than $n$,
   \item\label{prop:whisk properties2} the real dimension of $\Wcal(g)$ is also less than $n$,
   \item\label{prop:whisk properties3} the property $\Wcal(g) \neq \varnothing$ is preserved under real analytic $\Ascr$-equivalence.
\end{enumerate}
\end{proposition}
\begin{proof}
Observe that $g$ is finite because it is locally stable, see \cite[Proposition 7.1 and Remark 7.1]{Mond2020}. Hence $\im(g)$ is an analytic subset by Remmert finite map theorem (e.g., \cite[Corollary 1.68]{Greuel2007}). Since $g$ is defined by real analytic functions, we can find  real equations for $\im(g)$. 

Now, if $y\in\RR^p$ is a regular point of $\im(g)$, it can only have at most one pre-image by $g$ in $\CC^n$, which must therefore be a real point since conjugation acts on $g^{-1}(\RR^p)$ fiber-wise. This shows that $y\in\im(g^\RR)$ and, hence, $\Wcal(g)=\overline{\im(g)\cap\RR^p\setminus \im(g^\RR)}$ lies in the singular set of $\im(g)$. Since $\im(g)$ has dimension $n$, \cref{prop:whisk properties1} follows.

This also shows \cref{prop:whisk properties2}. Indeed, let $V_\CC$ be a complex variety of complex dimension $d$ that is given by real equations. Then, the real solutions $V_\RR\subseteq V_\CC\cap\RR^p$ have real dimension at most $d$. For non-singular points this is easy to show because the regular part of $V_\RR$ is contained in the regular part of $V_\CC$, since the rank of the Jacobian matrix at points of $V_\RR$ is equal in both cases. For the singular locus one applies an inductive argument.

Finally, to prove \cref{prop:whisk properties3}, let $\phi^\RR\colon(\R^n,S)\to(\R^n,S)$ and $\psi^\RR\colon(\R^p,0)\to(\R^p,0)$ be germs of real analytic diffeomorphisms and denote by $\phi$ and $\psi$ their complexifications.

Since diffeomorphisms in the source do not change the image, we have that $\im\big(\psi^\RR\circ g^\RR\circ \phi^\RR\big) = \psi^\RR\big(\im(g^\RR)\big)$.
Similarly, $\im\big(\psi\circ g \circ \phi\big) = \psi\big(\im(g)\big)$.
Now, since $\psi$ is equal to $\psi^\RR$ when restricted to $\R^p$, we have
\begin{equation*}
\begin{aligned}
\Wcal(\psi\circ g\circ \phi) &= \overline{\im(\psi\circ g\circ \phi)\cap \R^p \setminus \im(\psi^\RR\circ g^\RR\circ\phi^\RR)} \\
& = \overline{\psi\big(\im(g)\big)\cap \R^p \setminus \psi^\RR\big(\im(g^\RR)\big)}\\
& = \overline{\psi\big(\im(g)\cap \R^p\big) \setminus \psi^\RR\big(\im(g^\RR)\big)}\\
& = \psi\big(\Wcal(g)\big),
\end{aligned}
\end{equation*}
which gives the desired result.
\end{proof}

 %

\begin{lemma}\label{lem:embeddingdk}
Let $f:(\CC^n,S)\to(\CC^p,0)$ be an $\Ascr$-finite germ, $n<p\leq 2n$. Assume that $f$ has corank one or $(n,p)$ is in the nice dimensions. Then, $D^k(f)=\varnothing$ for $k\geq2$ if, and only if, $f$ is a germ of an embedding of $\CC^n$ to $\CC^p$. If $p>2n$, $D^k(f)=\varnothing$ for $k\geq2$ if, and only if, $S=\left\{*\right\}$.
\end{lemma}
\begin{proof}
The converse implication is obvious. For the direct implication necessarily $S=\left\{0\right\}$. If $f$ is unstable and we take a stabilisation, $F=(f_s,s)$, we provide a deformation of the multiple points of $f$, i.e., $D^k(F)$ is a deformation family of $D^k(f)$. Since emptiness would be preserved by deformations we can assume that $f$ is stable. For stable germs, this is precisely the statement of \cite[Proposition 9.8]{Mond2020}. Alternatively, one can show the result using an equality between the dimension of the local algebra of $f$, $\delta(f)$, and the maximal multiplicity of $f$ proven in \cite{Damon1976} (the statement works for stable germs by inspecting the classification in \cite{Mather1971}, cf. \cite[Theorem 9.1]{Mond2020}), since in our case $\delta(f)$ must be one and then $f$ must have maximal rank.
\end{proof}

\begin{lemma}[cf. {\cite[Proposition 9.5]{Mond2020}}]\label{cor:whisk diagonal}
Let $f$ be a multigerm as in \cref{lem:embeddingdk} and assume additionally that it is the complexification of a real germ $f^\RR$. If $f$ is $\Ascr$-equivalent to several smooth branches meeting in general position then $\Wcal(f)=\varnothing$. More generally, for any stable $f$ there is some $\sigma\neq\id$ so that (including empty intersections)
\[W^k(f)\cap D^k(f^\RR)\subseteq D^k(f)^\sigma, \text{ for all } \, k\geq 2.\]
\end{lemma}
\begin{proof}
We can work with $\Ascr$-classes by \cref{prop:whisk properties}.
For the first statement, we assume that each branch in the image is a hyperplane, after taking an $\Ascr$-equivalence if necessary. Since each hyperplane intersected with $\RR^p$ is a real hyperplane of dimension $n$, $\im(f)\cap\RR^p=\im(f^\RR)$ and $\Wcal(f)=\varnothing$.

For the second statement, observe that any ($\Ascr$-finite) monogerm $m:(\CC^n,0)\to(\CC^p,0)$ is such that $D^2(m)=\varnothing$ or $(0;0)\in D^2(m)^{(1\:2)}\neq\varnothing$. That $D^2(m)=\varnothing$ only happens for embeddings is shown in \cref{lem:embeddingdk}, which have empty whiskers.
 Hence, if $D^k(m)\neq\varnothing$ for some $k\geq2$, also $D^k(m)^\sigma\neq\varnothing$ for some $\sigma\neq\id$. Therefore, since any whisker of a multigerm is the union of whiskers of each branch, the lemma follows by noting that the intersection $\Wcal(f)\cap\im(f^\RR)$ is included in those points that are not embeddings nor normal crossings of embeddings (by the first statement), so they must have fixed points by some $\sigma\neq\id$ (cf. \cref{fig:whiskers}).
\end{proof}


The following threorem is a central piece in this work, although it is technical (see \cref{fig:grps1h2}).

\begin{theorem}\label{thm: icss splits}
Let $g$, $g^\RR$, $g\rvert_\Wcal$ and $g\rvert_{\RR^p}$ be as in \cref{def:whiskers,notation}. Then, for the different ICSS, we have that
\begin{enumerate}[label=(\roman*)]
	\item\label{it:icss splits1}  $E^1_{p,q}(g\rvert_{\RR^p};G)\cong E^1_{p,q}(g^\RR;G)\oplus E^1_{p,q}(g\rvert_{\Wcal};G),\quad \forall p\neq0$, and
   \item\label{it:icss splits2} $E^1_{p,q}(g^\RR;G)\hookrightarrow E^1_{p,q}(g\rvert_{\RR^p};G).$
\end{enumerate}
\end{theorem}
\begin{proof}
See \cref{sec: ICSS} for definitions. To show \cref{it:icss splits1}, since
\[E^1_{p,q}(\bullet;G)\coloneqq AH_q\big( D^{p+1}(\bullet);G \big)=H_q\big( C^\alt\big(D^{p+1}(\bullet);G\big), \partial|\big),\]
it is enough to show the result at the chain level (for $\ZZ$ coefficients, $k>1$):
\[ C^\alt_* \big(D^k(g\rvert_{\RR^p}) \big)\cong C^\alt_* \big(D^k(g^\RR )\big)\oplus C^\alt_* \big(D^k(g\rvert_{\Wcal}) \big).\]
Indeed, this holds since no cell lying in the diagonal of $\CC^n\times\dots\times\CC^n$ (i.e., the points fixed by some non-trivial permutation) is part of an alternating chain. This is shown in \cite[Theorem 2.12, Item (i)]{Houston2007}. Hence, after \cref{cor:whisk diagonal}, any alternating chain can be uniquely written as a sum of two alternating chains, one in $C^\alt (g^\RR )$ and the other in $C^\alt (g\rvert_{\Wcal} )$.

\cref{it:icss splits2} follows from \cref{it:icss splits1}, after observing that $E^1_{0,q}(g^\RR;G)\cong E^1_{0,q}(g\rvert_{\RR^p};G)$ for any $q$.
\end{proof}

\begin{figure}
	\centering
  \includegraphics[width=1\textwidth]{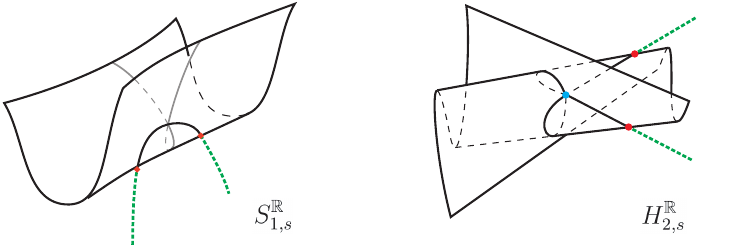}
\caption{Good real perturbations of the $S_1$ (left) and $H_2$ (right) singularities in Mond's classification of map germs from $\CC^2$ to $\CC^3$ \cite{Mond1985}, $S_{1,s}^\RR(x,y)=\big(x,y^2, y^3+y(x^2-s)\big)$ and $H^\RR_{2,s}(x,y)=\big(x,y^3-sy,xy+y^2(y^3-sy)\big)$.
}
\label{fig:grps1h2}
\end{figure}

The following corollary deals only with the free part of the homology. Its strengthen version to the homotopy type is given later in \cref{sec:homotopy}. However, we need this result previously. Furthermore, we deal with excellent real deformations in \cref{sec:otherdeformations}.

\begin{corollary}\label{cor:GRPisCGRP}
Let $f$ be an $\Ascr$-finite multigerm as in \cref{lem:embeddingdk}, then any good real perturbation of $f$ is also a complete good real perturbation. In particular, if $f_s^\RR$ is a good real perturbation of $f$,
\[ \chi_{Top}\big(\im(f_s)\big)=\chi_{Top}\big(\im(f_s^\RR)\big).\]
\end{corollary}
\begin{proof}
By left-right equivalence, we can consider that we work with the germ $f=f^\CC$, the complexification of $f^\RR$. Observe that $\im(f_s^\CC\rvert_{\RR^p})$ is the fixed point-set of the action of $\nicefrac{\ZZ}{2\ZZ}$ by conjugation (see \cref{notation}). Hence, we have the following chain of inequalities:
\begin{equation}\label{eq:chainbettis}
\begin{aligned}
\sum_i \beta_i\big(\im(f^\CC_s);\nicefrac{\ZZ}{2\ZZ}\big)&\geq \sum_i \beta_i\big(\im(f^\CC_s\rvert_{\RR^p});\nicefrac{\ZZ}{2\ZZ}\big)\\
&\geq\sum_i \beta_i\big(\im(f^\RR_s);\nicefrac{\ZZ}{2\ZZ}\big)\\
&\geq\sum_i \beta_i\big(\im(f^\CC_s);\nicefrac{\ZZ}{2\ZZ}\big).
\end{aligned}
\end{equation}
The first inequality follows by Smith theory; the second one is a consequence of the fact that the ICSS $E^1_{p,q}(f_s^\RR)$ is a direct summand of $E^1_{p,q}(f_s^\CC\rvert_{\RR^p})$ (by \cref{thm: icss splits}); and the last one follows from the fact that $f_S^\RR$ is a good real perturbation of $f^\CC$. Therefore, \cref{eq:chainbettis} must hold with equalities, showing the result.
%
%
\end{proof}

\section{Topology of multiple point spaces}\label{sec: Pruebas}

In this section we analyse the topology of the multiple point spaces $D^k$ of a good real deformation (recall \cref{sec: ICSS}). In order to do this, we need an extra ingredient that relates the fix point spaces by permutations $\sigma\in\Sigma_k$, see \cref{def:dksigma}, with the total space.

It is shown in \cite[Definition 2.8 and Proposition 3.4]{Houston2007} that the action of $\Sigma_k$ in our spaces $D^k$ is simplicial and if it fixes a simplex it is point-wise fixed. This is called a \textit{simplicially good action} and it is what is needed to apply the following.

\begin{lemma}[see {\cite[Proposition 2.4]{GimenezConejero2022c}}]\label{lem:formulachitau}
For an irreducible representation $\tau$ of a simplicially good action of $G$ in a simplicial complex $M$,
\[\chi_{\tau}(M)=\frac{1}{|G|}\sum_{\sigma\in G}\overline{\chi_\tau(\sigma)} \chi_{Top}(M^\sigma).\]
\end{lemma}

The notation of the lemma is clarified in \cite[Definition 2.3]{GimenezConejero2022c}. We only need to know that $\chi_{Top}$ is the usual topological characteristic and that we will use the lemma for $\tau=\Alt$, so $\chi_\tau(\sigma)=\sgn(\sigma)$ and 
\[\chi_\tau(\bullet)=\chi_\Alt(\bullet)\coloneqq\sum_i (-1)^{i}A\beta_i(\bullet).\]
Here $A\beta_i(\bullet) = \rk AH _i (\bullet)$.
The remaining of the section is a back and forth of arguments with the usual homology and the alternating homology, by means of the formula of \cref{lem:formulachitau}. We fix the notation $f_s^\RR$ for a good real deformation of an $\Ascr$-finite monogerm $f:(\CC^n,0)\to(\CC^p,0)$ of corank one, $n<p$. We keep using \cref{notation} as well. 

\begin{note}
Observe, however, that although the corank 1 hypothesis is used in \cref{prop:draftVII} to compute the Betti numbers of the multiple point spaces, it actually is not used in the proof of \cref{prop:draftXIII}.
\end{note}

\begin{proposition}\label{prop:draftVII}
For all $k\geq2$ and any element $\sigma\in\Sigma_k$,
\[ (-1)^{d_k^\sigma}\chi_{Top}\big(D^k(f_s)^\sigma\big)\geq(-1)^{d_k^\sigma}\chi_{Top}\big(D^k(f_s^\RR)^\sigma\big).\]
\end{proposition}
\begin{proof}
By Floyd's theorem,\cref{thm:ST_ineq},
\begin{equation}
\sum_{i=N}^\infty \beta_i\big(D^k(f_s)^\sigma\big)\geq\sum_{i=N}^\infty \beta_i\big(D^k(f_s^\RR)^\sigma\big)
\label{eq:draftVI}
\end{equation}
for all $N\geq 0$. 
Since $\beta_{0}\big(D^k(f_s)^\sigma\big) = 1$ and $\beta_{i}\big(D^k(f_s)^\sigma\big) = 0$ for $i\neq d_k^\sigma,0$  
by \cref{lem: dkgamma}, 
$$(-1)^{d_k^\sigma}\chi_{Top}\big(D^k(f_s)^\sigma\big) = \beta_{d_k^\sigma}\big(D^k(f_s)^\sigma\big) + (-1)^{d_k^\sigma}.$$
If $d_k^\sigma$ is even, then apply \cref{eq:draftVI} with $N=0$ to obtain the result:
$$ \beta_{d_k^\sigma}\big(D^k(f_s)^\sigma\big) + 1 \geq \sum_{i=0}^\infty \beta_i\big(D^k(f_s^\RR)^\sigma\big) \geq \chi_{Top}\big(D^k(f_s^\RR)^\sigma\big).$$
If $d_k^\sigma$ is odd, then \cref{eq:draftVI} with $N=1$ ensures
$$ \beta_{d_k^\sigma}\big(D^k(f_s)^\sigma\big)  \geq \sum_{i=1}^\infty \beta_i\big(D^k(f_s^\RR)^\sigma\big).$$
We can assume $\beta_0\big(D^k(f_s^\RR)^\sigma\big) \geq 1$, otherwise the result is trivial since the multiple point space would be empty. Hence, as desired,
\[ \beta_{d_k^\sigma}\big(D^k(f_s)^\sigma\big) - 1 \geq \sum_{i=1}^\infty \beta_i\big(D^k(f_s^\RR)^\sigma\big) - \beta_0\big(D^k(f_s^\RR)^\sigma\big)  \geq -\chi_{Top}\big(D^k(f_s^\RR)^\sigma\big).\qedhere\]
\end{proof}

\begin{theorem}\label{prop:draftXIII}
For all $k\geq2$ and all $i$,
\[A\beta_i\big(D^k(f_s)\big)=A\beta_i\big(D^k(f_s^\RR)\big).\]
\end{theorem}
\begin{proof}
We reason with the ICSSs $E^1_{p,q}(f_s;\QQ)$ and $E^1_{p,q}(f^\RR_s;\QQ)$, which have an inherently different distribution of non-trivial entries in its pages depending whether $p=n+1$ or not. We omit the coefficient to improve the readability. We recommend to follow the proof with \cref{fig: las mejores tablas del universo} at each step.
\newline

Assume $p> n+1$.
If, for some multiplicity $k$, we have that $A\beta_{d_k}\big(D^k(f_s)\big)=0$ then, since all the other $\Alt$-Betti numbers are also zero (see \cref{sec:complexdef}), we deduce that 
$A\beta_i\big(D^k(f_s^\RR)\big)=0$
for all $i$ by equivariant Smith Theory, \cref{thm:STtau}. 

Assume now that $k_0$ is the minimum $k$ such that $A\beta_{d_k}\big(D^k(f_s)\big)\neq0$. This appears in the ICSS as the rank of term $E^1_{k_0-1,d_{k_0}}(f_s)$. Hence, every column to its left is trivial by assumption. More precisely, every element $E^1_{a,b}(f_s)$ with $a<k_0-1$ and $b>0$ is zero, therefore so must be $E^1_{a,b}(f_s^\RR)$ with $a<k_0-1$ and $b>0$ by equivariant Smith Theory \cref{thm:STtau} (recall that $D^k(f_s)$ has no torsion in homology, \cref{lem: dkgamma}).

The only entries in $E^1_{*,*}$ that can kill the homology of $E^1_{k_0-1,d_{k_0}}(f_s^\RR)$ in some page of the spectral sequence are, for $r>0$, the $(k_0-1+r,d_{k_0}-r+1)$-entries (by taking images) or the $(k_0-1-r,d_{k_0}+r-1)$-entries (by taking kernels). The latter terms are trivial, as we have shown before. For the former, notice that they are always trivial because the elements in the column $k_0-1+r$ come from the spaces $D^{k_0+r}(f_s^\RR)$, which have dimension $d_{k_0+r}<d_{k_0}-r+1$ since $p>n+1$ (see \cref{lem: dkgamma}). In summary, we have that $E^\infty_{k_0-1,d_{k_0}}(f_s^\RR)\cong E^1_{k_0-1,d_{k_0}}(f_s^\RR)$. 

$E^\infty_{k_0-1,d_{k_0}}(f_s^\RR)$ is one of the possible terms contributing to the $(d_{k_0}+k_0-1)$-th Betti number of $\im(f_s^\RR)$. The other possible contributions come from the $(k_0-1+\ell,d_{k_0}-\ell)$-entries, $\ell\in\ZZ$; but for $\ell<0$ those entries in $E^1_{*,*}$ are zero by assumption on $k_0$ and for $\ell>0$ they are zero because, again, those columns are defined through the spaces $D^{k_0+\ell}(f_s^\RR)$, which have dimension $d_{k_0+\ell}<d_{k_0}-\ell$ since $p>n+1$ (see \cref{lem: dkgamma}).

Finally, for $f_s$ we already knew that $E^\infty_{k_0-1,d_{k_0}}(f_s)\cong E^1_{k_0-1,d_{k_0}}(f_s)$. Therefore, since $f_s^\RR$ is a good real perturbation of $f$ and $p>n+1$, we have that
\begin{align*}
    AH_{d_{k_0}}\big(D^{k_0}(f_s^\RR)\big)&\cong E^\infty_{k_0-1,d_{k_0}}(f_s^\RR)\cong H_{d_{k_0}+k_0-1}(\im(f_s^\RR)) \\ & \cong H_{d_{k_0}+k_0-1}(\im(f_s))  \cong E^\infty_{k_0-1,d_{k_0}}(f_s)\cong AH_{d_{k_0}}\big(D^{k_0}(f_s)\big).
\end{align*}
Now, equivariant Smith Theory \cref{thm:STtau} for the spaces $D^{k_0}(f_s)$ and $D^{k_0}(f_s^\RR)$ shows that the remaining $\Alt$-Betti numbers are zero, proving the result for $k_0$. We can argue by induction on $k$ using the same argument, because the previous non-trivial term is not going to influence the argument (assuming $p>n+1$). This shows the result in this case.
\newline

Assume $p= n+1$. The only (possible) non-trivial elements of the ICSS of $f_s$ are $E_{k-1,d_k}^1(f_s)$, they are in the same diagonal because $p=n+1$, and the sum of their ranks is equal to the $n$-th Betti number of $\im(f_s)$. For that reason, since $f_s^\RR$ is a good real perturbation, we must have
\begin{equation}  
\sum_k \rank E^1_{k+1,d_k}(f_s^\RR)\geq\sum_k \rank E^\infty_{k+1,d_k}(f_s^\RR)=  \sum_k \rank E^1_{k+1,d_k}(f_s).
\label{eq:GRPp=n}
\end{equation}
However, we also have that 
\begin{equation}
\rank E^1_{k+1,d_k}(f_s^\RR)\leq \rank E^1_{k+1,d_k}(f_s)
\label{eq:STp=n}
\end{equation} 
by equivariant Smith Theory, \cref{thm:STtau}. Then, we must have equality in both \cref{eq:GRPp=n,eq:STp=n}. The result follows now by applying equivariant Smith Theory again for the remaining $\Alt$-Betti numbers to show that they are zero.
\end{proof}

\begin{figure}
	\centering
  \includegraphics[width=0.90\textwidth]{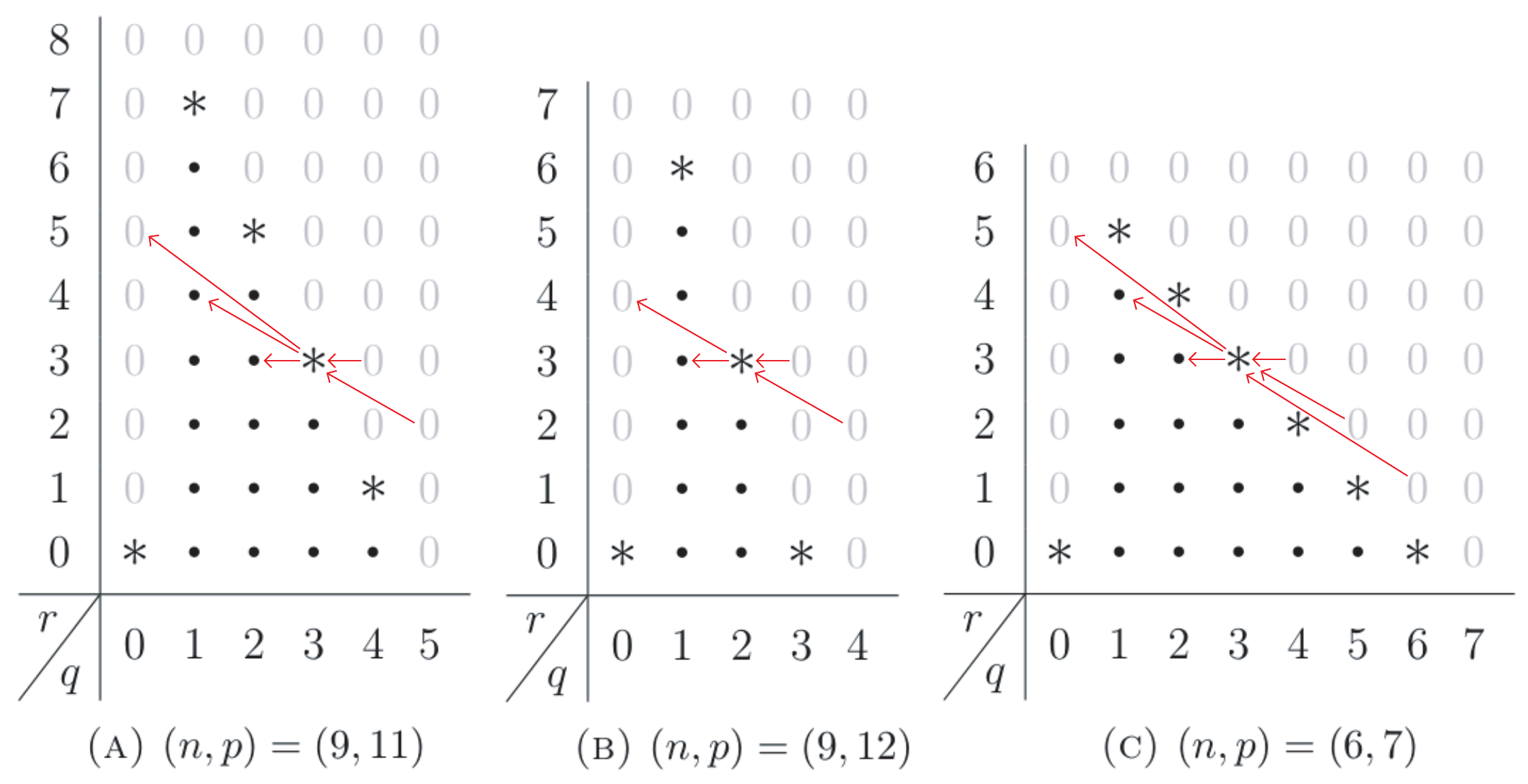}
\caption{Schematic of the ICSS $E^*_{q,r}(f_s^{(\RR)})$ where the arrows represent the different boundary operators at several pages that have target or source a given entry, \scalebox{1.3}{$\ast$} are the possibly non-zero entries in the complex case (also the corresponding dimension of the multiple point space in both cases) and \scalebox{0.6}{$\bullet$} are the homologies that are to be determined zero in the real case, for dimensions $(9,11)$, $(9,12)$ and $(6,7)$ from left to right.}
\label{fig: las mejores tablas del universo}
\end{figure}

By taking alternating sums, one has the following consequence.

\begin{corollary}\label{prop:draftX}
For all $k\geq2$,
\[ \chi_\Alt\big(D^k (f_s)\big)=\chi_\Alt\big(D^k(f_s^\RR)\big).\]
\end{corollary}

\begin{corollary}\label{prop:draftXI}
For all $k\geq2$ and any $\sigma\in\Sigma_k$,
\[ \chi_{Top}\big(D^k(f_s)^\sigma\big)=\chi_{Top}\big(D^k(f_s^\RR)^\sigma\big).\]
\end{corollary}
\begin{proof}
Using \cref{lem:formulachitau} for $\tau=\Alt$ and $M=D^k(f_s)$ or $M=D^k(f_s^\RR)$ and multiplying both equalities by $(-1)^{d_k}$,
\begin{equation}\label{eq:preprint}\begin{aligned}
(-1)^{d_k}\chi_\Alt\big(D^k(f_s)\big)&=\frac{1}{k!}\sum_\sigma(-1)^{d_k}\sgn(\sigma)\chi_{Top}\big(D^k(f_s)^\sigma\big),\text{ and}\\
(-1)^{d_k}\chi_\Alt\big(D^k(f_s^\RR)\big)&=\frac{1}{k!}\sum_\sigma(-1)^{d_k}\sgn(\sigma)\chi_{Top}\big(D^k(f_s^\RR)^\sigma\big).
\end{aligned}\end{equation}
The parity is convenient: $(-1)^{d_k}\sgn(\sigma)=(-1)^{d_k^\sigma}$ by \cite[Lemma 4.4]{GimenezConejero2022c} (or \cite[Lemma 6.3.2]{Robertothesis}).
By combining \cref{eq:preprint} and \cref{prop:draftX} altogether we obtain
\begin{equation}\label{eq:equal_chitop}
    \sum_{\sigma} (-1)^{d_k^\sigma}\chi_{Top}\big(D^k(f_s)^\sigma\big)  = \sum_{\sigma} (-1)^{d_k^\sigma}\chi_{Top}\big(D^k(f_s^\RR)^\sigma\big).
\end{equation}
Since \cref{prop:draftVII} ensures $(-1)^{d_k^\sigma}\chi_{Top}\big(D^k(f_s)^\sigma\big) \geq (-1)^{d_k^\sigma}\chi_{Top}\big(D^k(f_s^\RR)^\sigma\big)$ for each $\sigma$ and $k\geq 2$, if any of these inequalities were strict, the equality in \cref{eq:equal_chitop} would not hold.
\end{proof}

\begin{corollary}\label{prop:draftXII}
Assuming that $D^k(f_s^\RR)^\sigma\neq\varnothing$,
\[\begin{aligned}
\beta_i\big(D^k(f_s^\RR)^\sigma\big)&=0\text{ if } (-1)^{d_k^\sigma}\neq(-1)^i \text{ and }i>0,\\
\beta_0\big(D^k(f_s^\RR)^\sigma\big)&=1\text{ if } (-1)^{d_k^\sigma}\neq1.
\end{aligned}\]
\end{corollary}
\begin{proof}
Since
\begin{align*}
(-1)^{d_k^\sigma}\chi_{Top}\big(D^k(f_s)^\sigma\big)  =& \beta_{d_k^\sigma}\big(D^k(f_s)^\sigma\big) + (-1)^{d_k^\sigma},\text{ and}\\
(-1)^{d_k^\sigma}\chi_{Top}\big(D^k(f_s^\RR)^\sigma\big) =& \sum_{ i > 0\text{ even}}\beta_{d_k^\sigma - i}\big(D^k(f_s^\RR)^\sigma\big) - \sum_{ i > 0\text{ odd}}\beta_{d_k^\sigma - i}\big(D^k(f_s^\RR)^\sigma\big) \\ & + (-1)^{d_k^\sigma}\beta_{0}\big(D^k(f_s^\RR)^\sigma\big),
\end{align*}
by \cref{prop:draftXI} this implies
\begin{equation}\label{eq:main_draftXII}
\begin{split}
\sum_{ i > 0 \text{ odd}}\beta_{d_k^\sigma - i}\big(D^k(f_s^\RR)^\sigma\big)  =&  \sum_{ i > 0 \text{ even}}\beta_{d_k^\sigma - i}\big(D^k(f_s^\RR)^\sigma\big) - \beta_{d_k^\sigma}\big(D^k(f_s)^\sigma\big) \\
 & + (-1)^{d_k^\sigma}\big(\beta_{0}\big(D^k(f_s^\RR)^\sigma\big)-1\big).
 \end{split}
\end{equation}
Furthermore, by Smith theory,
\begin{equation}\label{eq:ST_draftXII}
\sum_{i\geq L}\beta_{i}\big(D^k(f_s)^\sigma\big) \geq \sum_{i \geq L}\beta_{i}\big(D^k(f_s^\RR)^\sigma\big),\ \forall L.
\end{equation}
Hence, if $d_k^\sigma$ is even, applying \cref{eq:ST_draftXII} for $L=0$ and using \cref{eq:main_draftXII},
\[\sum_{ i > 0 \text{ odd}}\beta_{d_k^\sigma - i}\big(D^k(f_s^\RR)^\sigma\big) \leq 0,\]
yielding $\beta_{d_k^\sigma - i}\big(D^k(f_s^\RR)^\sigma\big) = 0$ for all odd $i$.
If $d_k^\sigma$ is odd, the same argument for $M=1$ shows that
\[\sum_{ i > 0\text{ odd}}\beta_{d_k^\sigma - i}\big(D^k(f_s^\RR)^\sigma\big) \leq 1-\beta_{0}\big(D^k(f_s^\RR)^\sigma\big),\]
so we also get $\beta_{d_k^\sigma - i}\big(D^k(f_s^\RR)^\sigma\big) = 0$ for all odd $i$ and, moreover, $\beta_{0}\big(D^k(f_s^\RR)^\sigma\big)=1$.
\end{proof}

The following is a technical lemma we use later.

\begin{lemma}\label{lem:permutationsofRk}
Let $\Sigma_k$ be the group of permutations acting on $\RR^N\times\RR^k$ by permuting the last $k$ copies of $\RR$. Denote by $F$ the set of fixed points by some permutation. Then the complement of $F$ is a disjoint union of open sets $U_1,\dots,U_\kappa$. In these conditions, for every $i$ we have that $U_i\cap\sigma U_i\neq\varnothing$ if, and only if, $\sigma=\id\in\Sigma_k$.
\end{lemma}
\begin{proof}
We can assume that $N=0$, for positive $N$ the argument follows by projecting to $\RR^k$.

The set of fixed points of each tranposition $(i,j)$ divides $\Rbb^k$ in two half-spaces defined by the inequalities $x_{i} < x_j$ and $x_j < x_i$.
This shows that each $U_i$ in $\Rbb^k$ is uniquely determined by a total ordering over all variables,
$$x_{i_1} < x_{i_2} < \cdots < x_{i_k},$$
where  $x_{i_j} < x_{i_{j+1}}$ if $U_i$ is contained in the corresponding half-space of the permutation $(i_j,i_{j+1})$.
The only permutation $\sigma \in \Sigma_k$ that respects this ordering is the identity, therefore $U_i\cap\sigma U_i\neq\varnothing$ if and only if $\sigma = \id$.
\end{proof}

\begin{lemma}\label{lem:orbits}
If $A\beta_{d_k}\big(D^k(f_s)\big)>0$ and $d_k>0$, then there is only one orbit of connected components of $D^k(f_s^\RR)$, it has an odd number of connected components and, furthermore, there is only one component if $d_k$ is odd.
\end{lemma}
\begin{proof}
Let us fix the notation
\[ D\coloneqq D^k(f_s^\RR)=\bigsqcup_{i=1}^{\widetilde{\Theta}} \bigsqcup_{j=1}^{\widetilde{\theta_i}} \widetilde{D^i_j} \sqcup \bigsqcup_{i=1}^\Theta \bigsqcup_{j=1}^{\theta_i} D^i_j,\]
where $\widetilde{D}_j^i$ and $D^i_j$ are connected components of $D$, $\sqcup_{j=1}^{\widetilde{\theta_i}} \widetilde{D^i_j}$ are orbits of those connected components such that $\beta_{d_k}(\widetilde{D}_j^i)=0$, and $\sqcup_{j=1}^{\theta_i} D^i_j$ are orbits of those connected components with non-trivial $d_k$-th Betti number. In the following, by \textit{orbit} we always mean an orbit of \emph{connected components}.

Obviously, the number of orbits with non-trivial $d_k$-th Betti number, i.e., $\Theta$, is at least one, since $A\beta_{d_k}\big(D^k(f_s^\RR)\big)>0$. Observe also that none of the $\widetilde{\theta_i}$ or $\theta_i$ are equal to $k!$, otherwise
\[ A\beta_0(D)\geq A\beta_0 \left(\sqcup_{j=1}^{\theta_i} D^i_j\right)=1,\]
(or the equivalent statement with $\widetilde{D^i_j}$) because one could give an element in $AH_0$ by taking the image of a connected component by $\sum_\sigma\sgn(\sigma)\sigma$, contradicting \cref{sec:complexdef,prop:draftXIII}.
\newline

We show now that $\widetilde{\Theta}=0$ and $\Theta=1$. Notice that $\beta_0(D)=1$ if $d_k$ is odd by \cref{prop:draftXII}, so we can assume that $d_k$ is even. We prove the result using several times the fact that we cannot have more than one orbit with fixed points by a given transposition, say $\sigma=(a\ b)$, because then  $\beta_0(D^\sigma)$ has to be one (by \cref{prop:draftXII}) since $d_k^\sigma=d_k-1$ is odd.

It is clear that $\beta_0(D)>1$ if $\widetilde{\Theta}\neq0$ or $\Theta>1$. Recall that, in our case, the group $\Sigma_k$ acts by permuting the last $k$ copies of  $\RR$ in $\RR^{n-1+k}$ (see \cref{rem: eqs Dk}), in particular we can use \cref{lem:permutationsofRk}. Any orbit with an odd number of connected components has a component fixed by any given transposition. Indeed, a transposition either permutes pairs of components or fixes them, because it has order two. If there is more than one orbit with an odd number of connected components then we reach a contradiction with \cref{prop:draftXII}: If a transposition fixes a connected component (not necessarily point-wise) then that transposition also has fixed points, because a transposition acts by a reflection on a top-dimensional hyperplane and it splits the ambient space into two disjoint half-spaces (see \cref{lem:permutationsofRk}).  We show now that there are no orbits with even number of connected components, proving the result.

Given that any orbit has less than $k!$ elements, there is always at least one non-trivial permutation that fixes a connected component in every orbit by the orbit stabilizer theorem (again, not necessarily point-wise).  However, by \cref{lem:permutationsofRk}, if a permutation fixes a connected component then the connected component has to intersect one of the hyperplanes fixed by some transposition, so the component has fixed points by that transposition. Then, any orbit with an even number of components has necessarily a positive even number of components with fixed points by a transposition, contradicting \cref{prop:draftXII}.
\end{proof}

\begin{theorem}\label{lem:manifold}
If $D^k(f)$ is singular and $d_k>0$, then $D^k(f_s^\RR)^\sigma$ is a disjoint union of oriented closed manifolds of dimension $d_k^\sigma$ for all $\sigma$ such that $d_k^\sigma\geq0$.
\end{theorem}
\begin{proof} We can assume that $D^k(f_s^\RR)^\sigma\neq\varnothing$, otherwise the result is vacuously true.

We show first the case $\sigma=\id$. 

It is clear that $D^k(f_s)$ is the Milnor fiber of the ICIS $D^k(f)$, in particular it is given by a regular value of a set of (real) equations $g:\CC^N\to\CC^K$ (see \cref{rem: eqs Dk}).
Therefore, $D^k(f_s)$ is a manifold since it is the preimage of a regular value. Furthermore, by \cref{rem:propertiesDk,prop:draftXIII},
 we know that the top dimensional homology of $D^k(f_s^\RR)$ is not trivial and, by \cref{lem:orbits}, it is a disjoint union of copies of the same connected manifold. This implies that each of those copies is an oriented compact smooth manifold without boundary.
\newline

Now assume that $\sigma$ is a transposition. Observe that $D^k(f_s^\RR)^\sigma$ is the intersection of the manifold $D^k(f_s^\RR)$ with the hyperplane $H^\sigma$ of fixed points by $\sigma$. If we show that this is a transverse intersection then $D^k(f_s^\RR)^\sigma$ is also an oriented compact manifold without boundary (see, for example \cite[pp. 60 and 100]{Guillemin1974}). Indeed, by symmetry, the tangent space at any point of $D^k(f_s^\RR)$ cannot be contained in the hyperplane $H^\sigma$. To see that, observe first that the dimension of $D^k(f_s^\RR)^\sigma$ is $d_k^\sigma=d_k-1$, so $D^k(f_s^\RR)$ is not contained in $H^\sigma$. Moreover, for any point $p$ of $D^k(f_s^\RR)\cap H^\sigma$, there is a sequence of points $p_n$ in $D^k(f_s^\RR)\setminus H^\sigma$ converging to $p$. Hence, taking the limit $\ell$ of the segments from $p_n$ to $\sigma p_n$, which are orthogonal to $H^\sigma$, in the corresponding Grassmanian,  we see that $\ell$ is orthogonal to $H^\sigma$ and also contained in the tangent space $T_p D^k(f_s^\RR)$ by smoothness.
\newline

The general case follows the same steps, after noting that any linear space of fixed points by any permutation is an intersection of hyperplanes fixed by transpositions.
\end{proof}

\begin{corollary}\label{cor:oddDk}
If $D^k(f)^\sigma$ is singular of odd dimension then $\beta_i\big(D^k(f_s^\RR)^\sigma\big)=1$ if $i=0,d_k^\sigma$ and zero otherwise.
\end{corollary}
\begin{proof}
This follows by Poincaré duality, \cref{prop:draftXII,lem:manifold}.
\end{proof}

\begin{theorem}\label{thm:bettialtis1}
For every $k$ such that $d_k>0$,
\[A\beta_{d_k}\big( D^k(f_s)\big)=\begin{cases}
0 &\text{ if } D^k(f)=\varnothing \text{ or smooth},\\
1 &\text{ if } D^k(f) \text{ singular}.
\end{cases}\]
\end{theorem}
\begin{proof}
By \cref{rem:propertiesDk} we have the first case. If $D^k(f_s)$ is singular and $d_k>0$ then $A\beta_{d_k}\big(D^k(f_s)\big)>0$ by \cite[Theorem 4.8]{GimenezConejero2022c} (again, see \cref{rem:propertiesDk}), so we are in the conditions of \cref{lem:orbits} and we know that $D^k(f_s^\RR)$ has a unique orbit which is a union of manifolds without boundary, by \cref{lem:manifold}. Hence, denoting $D\coloneqq D^k(f_s^\RR)=D_1\sqcup\dots\sqcup D_\theta,$
\[H_n(D,\ZZ)\cong\left\langle \left[D_1\right],\dots,\left[D_\theta\right]\right\rangle .\]
It is clear then that $AH_n(D,\ZZ)$ must be generated by one element
\[ 
\left[D_1\right]\pm\dots\pm\left[D_\theta\right].\qedhere\]
\end{proof}

The following two lemmas give a clear indication of how the proof of \cref{main} below will go.

\begin{lemma}\label{lem:paramain}
Assume that $D^k(f)$ is singular and every space $D^k(f)^\sigma$ has dimension $d_k^\sigma\geq-1$. Then, $\mu\big(D^k(f)^\sigma\big)=1$ for all $\sigma$ such that $d_k^\sigma\geq0$.
\end{lemma}
\begin{proof}
Assume, for simplicity, that $d_k^\sigma\geq0$. Then, by \cite[Theorem 4.7]{GimenezConejero2022c} (which is derived from \cref{lem:formulachitau}),
\begin{align}
\mu^\Alt\big(D^k(f)\big)&=\frac{1}{k!}\left(\sum_{ d_k^\sigma\geq0}\mu\big(D^k(f)^\sigma\big)-\sum_{ d_k^\sigma<0}(-1)^{d_k^\sigma}\beta_0\big(D^k(f)^\sigma\big)\right)\\ \label{ec:paramain1}
 & =\frac{1}{k!}\sum_{ \sigma}\mu\big(D^k(f)^\sigma\big).
\end{align}
Recall that $D^k(f)^\sigma$ are ICIS by \cref{lem: dkgamma}. We also have by \cref{thm:bettialtis1} that the previous equation must be equal to one. Finally, since no space $D^k(f)^\sigma$ has negative expected dimension $d_k^\sigma$, all are singular by \cite[Theorem 4.8]{GimenezConejero2022c}, which implies that each summand in the previous equation is at least one. Since there are $k!$ terms, each Milnor number must be equal to one.

If $d_k^\sigma\geq-1$ the argument is similar, but now in \cref{ec:paramain1}, for each space $D^k(f)^\sigma$ with $d_k^\sigma=-1$, there is a term $-(-1)^{-1}\beta_0\big(D^k(f)^\sigma\big)=1$.
\end{proof}

\begin{lemma}\label{lem:paramain2}
If $D^k(f)$ is such that it is equivalent (i.e., by isomorphisms of local algebras) to a hypersurface, then $d_k^\sigma\geq -1$ for every $\sigma$.
\end{lemma}
\begin{proof}
The number of equations that defines $D^k(f)\subset \CC^{n+k-1}$ is $(p-n+1)(k-1)$ (recall \cref{rem: eqs Dk}). Since $D^k(f)$ is a hypersurface singularity, it is possible to find linear terms in the equations and eliminate all but one equation. More precisely, the matrix induced by the linear terms in the equations has rank $(p-n+1)(k-1)-1$. Furthermore, notice that we can assume that linear terms on the symmetric variables appear at most in one equation, after cancelling in other equations if necessary.

On the one hand, the smallest expected dimension $d_k^\sigma$ is attained with a maximal cycle in $\Sigma_k$ and it is
\[d_k-k+1=p-k(p-n+1)+1\]
Since we deal with monogerms, every space $D^k(f)^\sigma$ contains the origin. Therefore, if $k$ and $\sigma$ are such that $d_k^\sigma\leq-2$, then
\[p-k(p-n+1)+1\leq -2,\]
so
\[ k \geq \frac{p+3}{p-n+1}.\]
This would imply that the number of equations is at least
\[ (p-n+1)\left(\frac{p+3}{p-n+1}-1\right)=n+2. \]

On the other hand, each time we eliminate one equation we also eliminate one variable, there are $n-1$ non-symmetric variables and $k$ symmetric variables, but the symmetric variables appear as a linear term in only one equation. This shows that it is possible to eliminate at most $n$ equations, so $d_k^\sigma\geq -1$ as we wanted.
\end{proof}

\begin{theorem}\label{main}
If $d_k>0$ and $D^k(f)$ is singular, $\mu\big(D^k(f)\big)=1$.
\end{theorem}
\begin{proof}

Assume that $k$ is such that $D^k(f)$ is singular, so $\mu\big(D^k(f)\big)>0$. We show that necessarily $\mu\big(D^k(f)\big)=1$. If $d_k$ is odd then the result follows by \cref{cor:oddDk}.

Assume now that $d_k$ is even and positive. We can assume further that there are $\sigma$ such that $d_k^\sigma<-1$ (so $k\geq4$), otherwise the result also holds by \cref{lem:paramain}.

On the one hand, for a transposition $(i\; j)$, $D^k(f)^{(i\: j)}$ has Milnor number one by \cref{cor:oddDk}, since $d_k^{(i\: j)}$ is odd. By \cref{lem:milnorone}, $D^k(f)^{(i\: j)}$ is isomorphic to the $A_1$ hypersurface singularity.

The regular sequence that defines $D^k(f)^{(i\: j)}$ is that of $D^k(f)$ with an additional equation $y_i-y_j=0$, and all but one equations can be eliminated in $D^k(f)^{(i\: j)}$. Hence, so is the case in $D^k(f)$. Indeed, the term $y_i-y_j$ is linearly independent of all the other linear terms from the other equations, since the equations are symmetric in $y_1,\dots,y_k$ and $k\geq4$, so it is possible to eliminate all but one equation from the regular sequence of $D^k(f)$ as well. Now, the assumption of $d_k^\sigma<-1$ contradicts \cref{lem:paramain2}. The result follows from there.
\end{proof}

\begin{corollary}\label{cor:aftermain}
If $D^k(f)$ is singular and $d_k>0$, then $d_k^\sigma\geq-1$ for all $\sigma$. In particular, $d_k\geq k-2$.
\end{corollary}
\begin{proof}
By \cref{main}, $D^k(f)$ is an ICIS with Milnor number one, so it must be an $A_1$ singularity by \cref{lem:milnorone}. In particular, it is isomorphic to a hypersurface, so \cref{lem:paramain2} shows the result. The minimum $d_k^\sigma$ is obtained with a maximal cycle, in which case $d_k^\sigma=d_k-k+1$, so  $d_k\geq k-2$ as desired.
\end{proof}

\begin{corollary}\label{cor:aftermain2}
If some $D^\ell(f)$ is singular and $d_\ell>0$, then $D^k(f)=\varnothing$ for any $k$ such that $d_k< k-2$.
\end{corollary}
\begin{proof}
Indeed, by \cref{cor:aftermain}, all $D^k(f)$ with $d_k< k-2$ must be smooth or empty. Since $D^\ell(f)$ is singular, by \cite[Theorem 4.8]{GimenezConejero2022c}, all the multiple point spaces of multiplicity higher than $\ell$ are singular or empty.
\end{proof}

\section{Homotopy type}\label{sec:homotopy}

In this section we fix the notation of a germ $f:(\CC^n,0)\to(\CC^p,0)$ that is the complexification of a real map germ and it is $\Ascr$-finite, $n<p$ and $f$ is in the nice dimensions or of corank one. We also keep \cref{notation}.

This section is where we prove the conjecture of Cooper and Mond (for any corank), see \cite[Conjecture 4.1]{Cooper1998}, and generalize it to other dimensions but only in corank one.

\begin{conjecture}\label{conj:grp}
If $f:(\CC^n,0)\to(\CC^{n+1},0)$ admits a good real perturbation $f_s^\RR$, then the image of $f_s^\RR$ and of its complexification $f_s$ are homotopy equivalent.
\end{conjecture}

\subsection{The case $p=n+1$}
Let us address the case $p=n+1$ first. We use a lemma of Mond relating $f_s|_{\RR^p}$ and $f_s$, which is proved using well-known techniques of Morse theory.

\begin{lemma}[see {\cite[Lemma 2-2]{Mond1996}}]\label{lem:mond}
If $f_s^\RR$ is a good real perturbation of $f$, the inclusion of $\im f_s|_{\RR^p}$ in $\im f_s$ is a homotopy equivalence.
\end{lemma}

We also need the following theorem by Houston, which gives a general result on fundamental groups.

\begin{lemma}[see {\cite[4.19]{Houston1997}}]\label{lem:houstonpi1}
Let $g:X\to Y$ be a proper and finite surjective stratified submersion. Assume that $X$ is path-connected and that there is a point $y\in g(X)$ with only one preimage $x$. If $H_0^\Alt\big(D^2(g)\big)\cong0$, then
\[g^*: \pi_1(X,x)\to \pi_1(Y,y)\]
is surjective.
\end{lemma}

\begin{theorem}\label{thm:htequin+1}
If $f_s^\RR$ is a good real perturbation of $f$, then $\im f_s^\RR$ is a deformation retract of $\im f_s$.
\end{theorem}
\begin{proof}
The statement is equivalent to having a homotopy equivalence between $\im f_s^\RR$ and $\im f_s$ by, for example, \cite[Corollary 0.20]{Hatcher2002} (see also \cite[Proposition 0.16]{Hatcher2002}). Hence, we focus on homotopy equivalences.
\newline

Observe that the case $n=1$ is trivial after \cref{lem:mond}, since $f_s^\RR=f_s|_\RR$ because there are no whiskers in dimension one by \cref{cor:whisk diagonal} (also \cref{prop:whisk properties}).

First, we show that the inclusion induces an isomorphism in homology,
\[H_*(\im f_s^\RR)\cong H_*( \im f_s).\]
Since the inclusion induces a homotopy equivalence between $\im f_s|_{\RR^{n+1}}$ and $\im f_s$ by \cref{lem:mond}, it suffices to show that the inclusion of $\im f_s^\RR$ in $\im f_s|_{\RR^{n+1}}$ is an isomorphism in homology. Since all the Betti numbers coincide (cf. \cref{cor:GRPisCGRP}) and there is no torsion in $H_*(\im f_s|_{\RR^{n+1}})$, we must prove that there is no torsion in $H_*(\im f_s^\RR)$ either. Indeed, we show that there is no extension problem in the spectral sequence. Recall that a spectral sequence could converge to non-isomorphic modules, it only determines a bigraded module obtained from a filtration,
\[ E^0_{p,q}\cong\frac{F^{p} H_{p+q}}{F^{p-1} H_{p+q}}.  \]
Hence, in good circumstances as our case (i.e., bounded filtration), one could obtain a limit $\hat{H}_*$ inductively from the (in general, non-unique) extensions
\[\begin{aligned}
0\to& E^{\infty}_{*,*}\to F^{0} \hat{H}_*\to 0\to 0, \text{ and}\\
0\to& E^\infty_{*,*} \to F^{p} \hat{H}_* \to F^{p-1} \hat{H}_* \to 0.
\end{aligned}\]
In corank one we can proceed from our results, we do it in \cref{subsec:p>n+1}, but we can give an argument for any corank when $p=n+1$. Since the ICSS of $f_s^\RR$ injects into the ICSS of $f_s|_{\RR^{n+1}}$ by the inclusion, see \cref{thm: icss splits}, and the reduced homology of $\im f_s|_{\RR^{n+1}}$ is trivial except in dimension $n$, where it is free, the limit of the ICSS of $f_s|_{\RR^{n+1}}$ must be trivial between the $n$-th and the $0$-th diagonals. Therefore, so is the case in the ICSS of $f_s^\RR$ and all the extensions to determine the limit are trivial in that range. Lastly, regarding $n$-th homology, since $\im f_s|_{\RR^{n+1}}=\im f_s^\RR \cup \im f_s|_{\Wcal}$ and $\im f_s|_{\Wcal}$ has dimension less than $n$ (by \cref{prop:whisk properties}), the $n$-th homologies must be isomorphic, hence free.

Finally, we only have to show that $\im f_s^\RR$ is simply connected, by Whitehead's theorem (e.g., Theorem 4.5 in \cite{Hatcher2002}). For this, we use Houston's result \cref{lem:houstonpi1} for $f_s^\RR$ from a topological ball $U\subset \RR^n$ to its image. The stratification by stable types provides the finite proper stratified submersion, and we also have points with only one preimage. It remains to show the condition on $D^2(f_s^\RR)$. Indeed, it follows from \cref{prop:draftXIII} (which holds in any corank) and the fact that $d_2>0$ (recall also the discussion in the beginning of \cref{sec:complexdef}).
\end{proof}

\subsection{A small correction}\label{sec:houston}
In \cite[Theorem 5.5]{Houston2005}, Houston showed a proof of the conjecture in dimensions $(2,3)$ that uses the following lemma from \cite{Houston1997}.

\begin{lemma}[see {\cite[Lemma 2.6]{Houston1997}}]
Suppose $\Sigma_k$ acts on a subcomplex $Y$ of $X^k$, and that $Y$ is the orbit $\Sigma_k Z$ of some path-connected $Z$.
\begin{enumerate}[label=(\roman*)]
	\item If $Z\cap Diag(\Sigma_k)=\varnothing$ and $\sigma Z\cap Z=\varnothing$ for all $\sigma\in\Sigma_k$, then $AH_0(Y)\cong\ZZ$.
   \item\label{it:iicounterexample} If $Z\cap Diag(\Sigma_k)=\varnothing$ and $\sigma Z\cap Z\neq\varnothing$ for some $\sigma\in\Sigma_k$, then $AH_0(Y)\cong\nicefrac{\ZZ}{2\ZZ}$.
   \item If $Z\cap Diag(\Sigma_k)\neq\varnothing$, then $AH_0(Y)\cong0$.
\end{enumerate}
\end{lemma}

The reason to use this lemma is to prove that there is no torsion in $AH_*$ when trying to prove the conjecture. Unfortunately, this lemma is false, as we show now with a counterexample. However, it does not seem to affect the integrity of \cite{Houston1997} since what is actually used is \cite[Remark 2.7]{Houston1997}, which is true. On the other hand, it seems plausible that the mistake in \cite[Theorem 5.5]{Houston2005} can be fixed by adding a lemma that studies the torsion of $AH_*$ in $CW$-complexes of dimension one.

\begin{example}\label{ex:triangles}
To simplify the argument, we set $a\coloneqq -1-i$, $b\coloneqq 1+i$ and $c\coloneqq -i$. Take $X^k=\CC^3$ and $Z$ the segment from the point $(a,b,c)$ to $(b,c,a)$, so that $Y=\Sigma_3 Z$. It is easy to see that $Z$ satisfies \cref{it:iicounterexample} above. Also, it is not hard to show that $\partial C^\Alt_1(Y)=0$ (see \cref{fig:triangles}), hence, $AH_0(Y)\cong C^\Alt_0(Y)\cong\ZZ\ncong\nicefrac{\ZZ}{2\ZZ}$.
\end{example}

\begin{figure}
	\centering
		\includegraphics[width=0.75\textwidth]{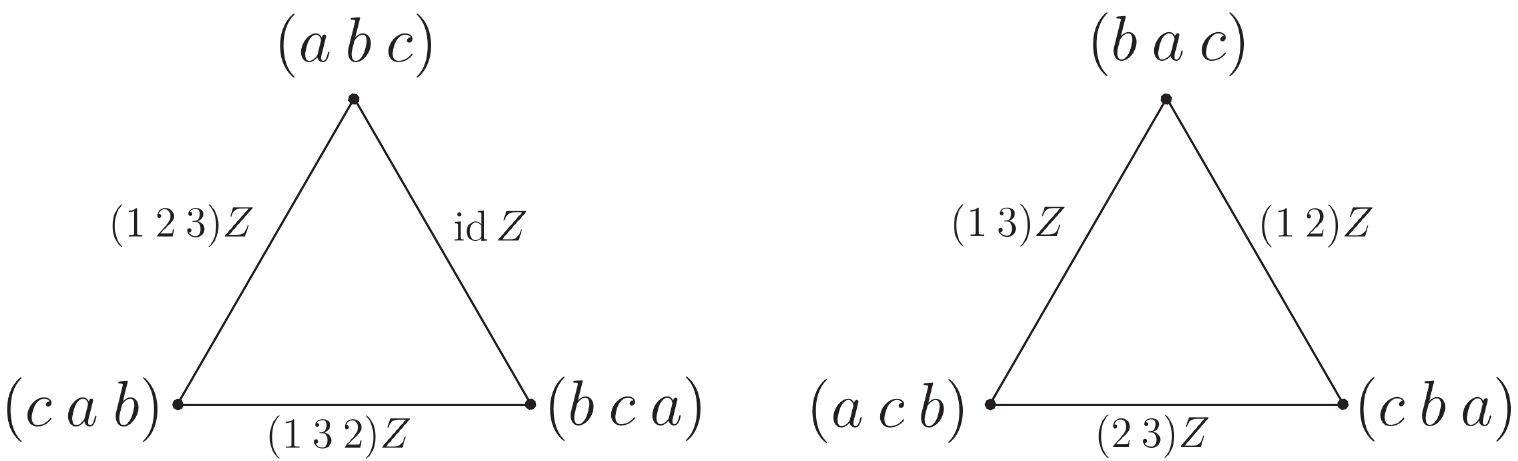}
	\caption{Schematic of \cref{ex:triangles}.}
	\label{fig:triangles}
\end{figure}

\subsection{The case $p>n+1$}\label{subsec:p>n+1}

The idea for the case $p>n+1$ is the same as the case $p=n+1$, with the exception that we do not have \cref{lem:mond}, nor any way of recovering it, since the images are not hypersurfaces any more. For this reason, we turn to the ICSS $E^1_{*,*}(f_s^\RR)$ and $E^1_{*,*}(f_s)$ from \cref{sec: ICSS}. This is also the technical reason we have to restrict to corank one map germs, in contrast with the previous case.

\begin{lemma}\label{lem:AHcornakone}
If $f_s^\RR$ is a good real perturbation of a corank one $\Ascr$-finite map  germ $f:(\CC^n,0)\to(\CC^p,0)$, $n<p$, then the inclusion of multiple point spaces induces the isomorphism
\[AH_*\big(D^k(f_s)\big)\cong AH_*\big(D^k(f_s^\RR)\big).\]
\end{lemma}
\begin{proof}
We assume that $k>1$ and $d_k>0$, otherwise the result is trivial. We know from \cref{thm:iciscase,thm:bettialtis1,main} that $D^k(f_s)$ deformation retracts to $D^k(f_s^\RR)$. Indeed, the pair $\big(D^k(f_s),D^k(f_s^\RR)\big)$ is homeomorphic to the pair $(TS^{d_k} ,S^{d_k} )$, which can be seen from a standard argument (recall that $D^k(f_s)$ is an $A_1$ singularity, by \cref{main}, so $D^k(f_s)$ is a \textit{complex sphere}). The retraction can be obtained from the flow of a smooth vector field, simply by retracting the fibers of the tangent bundle $TS^{d_k}$, which can then be $\Sigma_k$-averaged to make it $\Sigma_k$-equivariant. This idea of $G$-averaging vector fields was used before in \cite[Section 2.2]{Houston1997}. If the retraction is $\Sigma_k$-equivariant the homologies coincide as desired.
\end{proof}

\begin{theorem}\label{thm:htequip}
In the conditions of \cref{lem:AHcornakone}, $\im f_s^\RR$ is a deformation retract of $\im f_s$.
\end{theorem}
\begin{proof}
The proof follows the same steps of the proof of \cref{thm:htequin+1}. However, we need a different argument to show that the inclusion induces the isomorphism
\[H_*(\im f_s^\RR)\cong H_*( \im f_s).\]
We show now a topological argument, see also an argument that uses homological algebra in \cref{rem:homologicalalgebra}.
\newline

On the one hand, the inclusion $incl_{D^k}:D^k(f_s^\RR)\hookrightarrow D^k(\im f_s)$ induces an isomorphism of the corresponding spectral sequences at the first page by \cref{lem:AHcornakone}. 

On the other hand, this isomorphism of spectral sequences is \textsl{compatible} with $incl_f^*:H_*(\im f_s^\RR)\rightarrow H_*(\im f_s)$ in the sense of \cite[Comparison Theorem 5.2.12]{Weibel1994}, which implies that $incl_f^*$ is actually an isomorphism. 
To show this, we have to prove that $incl_f^*$ preserves the filtration and that its induced map
\begin{equation}\label{eq:inclf}
incl_f^*: \frac{F_pH_n(\im f_s^\RR)}{F_{p-1}H_n(\im f_s^\RR)}\rightarrow \frac{F_pH_n(\im f_s)}{F_{p-1}H_n(\im f_s)} 
\end{equation}
correspond to the map on the $E^\infty_{*,*}$ pages
\begin{equation}\label{eq:incldk}
incl_{D^k}^\infty: E^\infty_{*,*}(f_s^\RR) \rightarrow E^\infty_{*,*}(f_s) 
\end{equation}
that is induced from $incl_{D^k}$. 

Both spectral sequences collapse at the first page, so $E^1_{*,*}\cong E^\infty_{*,*}$ in both cases and the map $incl_{D^k}^\infty$ is the isomorphism from \cref{lem:AHcornakone}, which makes following the argument easier. First, observe that the inclusion $incl_f$ induces the inclusion $incl_{D^k}$, since we are simply taking the real part. In turn, $incl_f$ also induces the inclusion 
\begin{equation}\label{eq:chainincl}
incl_C:C^\Alt_*\big(D^k(f_s^\RR)\big)\hookrightarrow C^\Alt_*\big(D^k(f_s)\big).
\end{equation}
Since we can reduce to an inclusion at the chain level, checking the conditions we need is a lengthy but simple algebraic verification. Indeed, every morphism we need in the argument is induced from an inclusion. Possibly, the best way to verify the required conditions is by following the proof of \cite[Theorem 5.4]{Houston2007} having in mind that, with Houston's notation, $H$ is trivial and $\widetilde{X}$ is empty (cf. \cite{Goryunov1995}). There, a semi-simplicial resolution $W$ with a natural filtration (which gives the spectral sequence) is used. The inclusion also induces a map in these semi-simplicial resolutions, and it is easy to check that it respects the filtration. The correspondence between \cref{eq:inclf,eq:incldk} also follows from the chain morphisms, see an explicit form of the correspondence (also induced from inclusions) in the proof of \cite[Theorem 2.6]{McCleary2001}.
\end{proof}

\begin{remark}\label{rem:homologicalalgebra}
It is also possible to give a shorter proof of the preceding theorem using homological algebra. From the (filtered) chain morphism $incl_C$ in \cref{eq:chainincl}, we follow \cite[Exercise 5.4.4]{Weibel1994}. First, we construct the mapping cone $cone(incl_C)$ with a filtration such that the corresponding spectral sequence $E^r_{*,*}\big(cone(incl_C)\big)$ is the mapping cone of the morphism of spectral sequences $E^1_{*,*}(f_s^\RR)\rightarrow E^1_{*,*}(f_s)$. This gives a long exact sequence between these three spectral sequences. However, since $E^1_{*,*}(f_s^\RR)\rightarrow E^1_{*,*}(f_s)$ is an isomorphism by \cref{lem:AHcornakone}, $E^\infty_{*,*}\big(cone(incl_c)\big)$ is zero, which implies that $cone(incl_C)$ has trivial homology and that $incl_C$ is a quasi-isomorphism as we wanted.
\end{remark}

\section{Other real deformations}\label{sec:otherdeformations}
In this section, we consider $\Ascr$-finite map germs $f:(\CC^n,0)\to(\CC^p,0)$ of corank one, with $n<p$.
\newline

Going back to \cref{def:grp}, we have not dealt with excellent real deformations. The reason is simple, the easiest argument can be given after \cref{lem:AHcornakone}, were we indeed show that the multiple points of good real deformations $D^k(f_s^\RR)$, in corank one, are an \textsl{equivariant} deformation retract of the multiple points of the complex deformation $D^k(f_s)$. After that, it is obvious how to use the following theorem of Houston.

\begin{theorem}[see {\cite[Theorem 3.2]{Houston2002a}}]\label{thm:houstonMk}
Suppose that $g_i:X_i\to Y_i$, $i=1,2$, are finite and proper continuous maps for which the ICSS exist. Assume that there are continuous maps $\phi,\psi$ making the following diagram commutative,
\[\begin{tikzcd}
	{X_1} & {Y_1} \\
	{X_2} & {Y_2}
	\arrow["{g_1}", from=1-1, to=1-2]
	\arrow["\phi"', from=1-1, to=2-1]
	\arrow["\psi", from=1-2, to=2-2]
	\arrow["{g_2}", from=2-1, to=2-2]
\end{tikzcd}.\]
Then, if the induced map $\phi^k:D^k(g_1)\to D^k(g_2)$ from $\phi$ is an equivariant homotopy equivalence for all $k$, the restriction $\phi|:M_r(g_1)\to M_r(g_2)$ induces an isomorphism in integer homology for all $r$.
\end{theorem}

\begin{proposition}\label{prop:good=excellent}
Good real perturbations are also excellent real perturbations.
\end{proposition}
\begin{proof}
This follows from \cref{lem:AHcornakone} and using \cref{thm:houstonMk} with the inclusion.
\end{proof}

We turn now to M-perturbations (recall \cref{def:Mperturbation,fig:grps1h2}). We need a previous key observation.

\begin{remark}\label{rem:typesdim0}
For a fixed pair of dimensions, the stable singularity types of dimension zero (in corank one) are in bijection with multiple point spaces $D^k(\bullet)^\sigma$ with $d_k^\sigma=0$. This can easily seen from the classification given, for example, in \cite[Proposition 3.4]{OsetSinha2022}: all the stable singularities of dimension zero are a generalization of the cross-cap or other stable multi-germs of higher dimension intersecting in a convenient way (with isosingular locus in general position), which determine a permutation $\sigma$ and a $k$ such that $d_k^\sigma=0$; and vice-versa.
\end{remark}

\begin{proposition}
Good real perturbations are also $M$-perturbations.
\end{proposition}
\begin{proof}
We need to show that any good real perturbation has the same stable singularities of dimension zero. By \cref{prop:draftXI},
\[ \chi_{Top}\big(D^k(f_s)^\sigma\big)=\chi_{Top}\big(D^k(f_s^\RR)^\sigma\big).\]
This, when $d_k^\sigma=0$, and \cref{rem:typesdim0} show the result.
\end{proof}

\begin{remark}
It is known that the converse of the previous result is not true: not every germ that has an $M$-perturbation has a good real perturbation. It is shown in \cite[Theorem 2.1]{Rieger2008a} that all germs $(\CC^n,0)\to(\CC^{n+1},0)$ of corank one that are simple, $n\neq 4$, have an $M$-deformation, but not all of them have good real perturbations: see \cite{Marar1996} for $n=2$ and  \cref{sec:C3C4} below for $n=3$.
\end{remark}

\section{Germs from $\CC^3$ to $\CC^4$}\label{sec:C3C4}

Our results allows us to know what germs do not have good real picture in a practical way, and for a given map germ it is easy to find its good real perturbation if it has it. We show this with the classification of corank one germs from $(\CC^3,0)$ to $(\CC^4,0)$ shown in \cite{Houston1999a}, see \cref{tab:simplec3c4,tab:nonsimplec3c4}.

\begin{table}
\scalebox{0.8}{
\begin{tabular}{l|l|cc|cc|c}
Name                           &$f(x,y,z)=(x,y,\bullet,\bullet)   $                          & $\mu(D^2)$ & $\mu(D^3)$ & $\Ascr_e$-codim         & $\mu_I$         & Condition                                          \\ \hline
 $A_k$     & $\big(x,y,z^2,z(z^2+x^2+y^{k+1})\big)$    & $k    $      & -          & \multicolumn{2}{c|}{$k$}                     & $k\geq1$                                           \\
 $D_k$     & $\big(x,y,z^2,z(z^2+x^2y+y^{k-1})\big)$   &$ k     $    & -          & \multicolumn{2}{c|}{$k$}                     & $k\geq4$                                           \\
 $E_6$     & $\big(x,y,z^2,z(z^2+x^3+y^4)\big)$        &$ 6     $     & -          & \multicolumn{2}{c|}{$6$}                     & -                                                  \\
 $E_7$     & $\big(x,y,z^2,z(z^2+x^3+xy^3)\big)$       &$ 7    $     & -          & \multicolumn{2}{c|}{$7$}                     & -                                                  \\
 $E_8$     & $\big(x,y,z^2,z(z^2+x^3+y^5)\big)$        &$ 8     $     & -          & \multicolumn{2}{c|}{$8$}                     & -                                                  \\
 $B_k$     & $\big(x,y,z^2,z(x^2+y^2+z^{2k})\big)$     &$ 2k-1  $     & -          & \multicolumn{2}{c|}{$k$}                     & $k\geq2$                                           \\
 $C_k$     & $\big(x,y,z^2,z(x^2+yz^2+y^k)\big)$       &$ k+1     $   & -          & \multicolumn{2}{c|}{$k$}                     & $k\geq3$                                           \\
 $F_4$     & $\big(x,y,z^2,z(x^2+y^3+z^4)\big)$        &$ 6      $    & -          & \multicolumn{2}{c|}{$4$}                     & -                                                  \\
 $P_3^k$   & $(x,y,yz+z^6+z^{3k+2},xz+z^3)$            &$ 0        $  &$ 6k+1 $      &$ k+1         $            & $k+2       $      & $k\geq2$                                           \\
 $P_4^1$   & $(x,y,yz+z^7+z^8,xz+z^3)$                 &$ 0       $   &$ 16 $        & 4                       & 5               & -                                                  \\
 $P_k$     & $(x,y,yz+z^{k+3},xz+z^3)$                 &$ 0      $    &$ k^2   $     & \multicolumn{2}{c|}{$\frac{1}{6}(k+1)(k+2)$} & $k\geq1$, $3\nmid k$                               \\
 $Q_k$     & $(x,y,xz+yz^2,z^3+y^kz)$                  & $k-1     $   & $1$          & \multicolumn{2}{c|}{$k$}                     & $k\geq2$                                           \\
 $R_k$     & $(x,y,xz+z^3,yz^2+z^4+z^{2k-1})$          & $2k-3    $   &$ 4   $       &$ k  $                     & $k+1 $            & $k\geq3$                                           \\
 $S_{j,k}$ & $(x,y,xz+y^2z^2+z^{3j+2},z^3+y^kz)$       &$ k-1   $     & $6j+1   $    & $k+j$                     & $k+j+1 $          & $j\geq1,k\geq2$      
\end{tabular}
}
\caption{Simple corank one singularities, $(\CC^3,0)\to(\CC^4,0)$. None have quadruple points.}
\label{tab:simplec3c4}
\end{table}

On the one hand, we know that a germ that has a good real perturbation must have multiple point spaces so that $\mu\big(D^k\big)=0,1$ if the dimension $d_k$ is positive, by \cref{main}. In these dimensions $d_2=2$ and $d_3=1$, so we know that from the list of simple singularities in \cref{tab:simplec3c4} only $A_1, P_1,$ and $Q_2$ may have good real perturbations. Indeed, we already know that $A_1$ and $P_1$ have good real perturbations because they are corank one germs with $\Ascr_e$-codimension one in dimensions $(n,n+1)$, a class that is known to have good real perturbations by \cite[Theorem 7.3]{Cooper2002}. The germ $Q_2$ presents a behaviour that is not possible to see in lower dimensions, it has two multiple point spaces of positive dimension with Milnor number one. 

On the other hand, we know from the combination of \cref{thm:iciscase,prop:draftXIII}, that the spaces $D^2$ and $D^3$ of a good real perturbation must be spheres. All things together makes very easy to find the good real perturbation:
\[Q_{2,s}^\RR(x,y,z)=(x,y, xz+yz^2,z^3+y^2z-sz).\]
Observe that the equations of the respective multiple point spaces are given by divided differences (see \cref{rem: eqs Dk}):
\[\begin{aligned}
D^2(Q_{2,s}^{(\RR)})=&\Vcal\left( x+y(z_1+z_2);\ z_1^2+z_1z_2+z_2^2+y^2-s \right)\subset (\CC^4,0), \\
D^3(Q_{2,s}^{(\RR)})=&\Vcal\left( x+y(z_1+z_2);\ z_1^2+z_1z_2+z_2^2+y^2-s;\ y;\ z_1+z_2+z_3  \right)\subset (\CC^5,0).
\end{aligned}\]
To show that $Q_{2,s}^{\RR}$ gives a good real perturbation, we just check that the alternating homology of both spaces coincides in the complex and in the real case.
It is easy to see that they are smooth in the complex case, so it is indeed a stable perturbation by the Marar-Mond criterion \cref{lem: dkgamma}. 
Moreover, both complex spaces are the Milnor fiber of a Morse singularity, hence have the homotopy type of an $S^2$ and an $S^1$ respectively. It is easy to see that the real spaces are, respectively, an $S^2$ and an $S^1$ after a change of coordinates.
Finally, to show that the alternating homology is the same, we can use \cref{lem:formulachitau} (cf. \cref{eq:preprint} and \cite[Theorem 4.7]{GimenezConejero2022c}):
\[(-1)^{d_k}\chi_\Alt(D^k)=\frac{1}{k!}\sum_\sigma(-1)^{d_k}\sgn(\sigma)\chi_{Top}(D^{k,\sigma}).\]
It is easy to see that $D^2$ fixed by $(1\; 2)$ is also a sphere of dimension $d_2^{(1\: 2)}=1$ in the real case, we simply have to add the equation $z_1-z_2=0$. Hence,
\[A\beta_2(D^2)=\frac{2+0}{2}=1,\]
in the real and in the complex case (which can be deduced from the Milnor numbers as well).

 Something similar happens for $D^3$ fixed by a transposition $(a\; b)$, it is two points. Since $d_3^{(a\: b\: c)}=-1$, the space fixed by $(a\; b\; c)$ is empty. Hence,
\[A\beta_2(D^3)=\frac{0+(2+2+2)+(0+0)}{3!}=1.\]
This shows that it is a good real deformations and, in general, how to find good real perturbations of a given map germ that is a good candidate after \cref{main}.

\begin{table}[h]
\scalebox{0.8}{
\begin{tabular}{ll|cc|cc|c}
Name                                       & $f(x,y,z)=(x,y,f_1,f_2)$                     & $\mu(D^2)$         & $\mu(D^3)$          & $\Ascr_e$-codim    & $\mu_I$            & Condition                                                           \\ \hline
 &&&&&&\\[-10pt]
 \multirow{2}{*}{I }    & $f_1=yz+xz^3+z^5+az^7$           & \multirow{2}{*}{0} & \multirow{2}{*}{13} & \multirow{2}{*}{5} & \multirow{2}{*}{6} & \multirow{2}{*}{$a\neq b$}                                          \\
                       & $f_2=xz+z^4+bz^6$               &                    &                     &                    &                    &                                                                     \\[3pt]
 \multirow{2}{*}{II }   & $f_1=yz+xz^3+az^6+z^7+bz^8+cz^9$ & \multirow{2}{*}{0} & \multirow{2}{*}{25} & \multirow{2}{*}{7} & \multirow{2}{*}{9} & \multirow{2}{*}{Generic$^\dagger$}                                          \\
                       & $f_2=xz+z^4$                     &                    &                     &                    &                    &                                                                     \\[3pt]
 \multirow{2}{*}{III }  & $f_1=yz+z^5+z^6+az^7$            & \multirow{2}{*}{0} & \multirow{2}{*}{13} & \multirow{2}{*}{5} & \multirow{2}{*}{6} & \multirow{2}{*}{$a\neq1$}                                           \\
                       & $f_2=xz+z^4$                    &                    &                     &                    &                    &                                                                     \\[3pt]
 \multirow{2}{*}{IV }   & $f_1=yz+z^5+az^7$                & \multirow{2}{*}{0} & \multirow{2}{*}{13} & \multirow{2}{*}{5} & \multirow{2}{*}{6} & \multirow{2}{*}{$a\neq1$}                                           \\
                       & $f_2=xz+z^4+z^6$                &                    &                     &                    &                    &                                                                     \\[3pt]
 \multirow{2}{*}{V }    & $f_1=xz+z^5+ay^3z^2+y^4z^2$      & \multirow{2}{*}{1} & \multirow{2}{*}{13} & \multirow{2}{*}{5} & \multirow{2}{*}{6} & \multirow{2}{*}{$\forall a$}                                        \\
                       & $f_2=z^3+y^2z$                  &                    &                     &                    &                    &                                                                     \\[3pt]
 \multirow{2}{*}{VI }   & $f_1=xz+z^3$                     & \multirow{2}{*}{3} & \multirow{2}{*}{13} & \multirow{2}{*}{4} & \multirow{2}{*}{6} & \multirow{2}{*}{$a\neq1$}                                           \\
                       & $f_2=yz^2+z^5+z^6+az^7$         &                    &                     &                    &                    &                                                                     \\[3pt]
 \multirow{2}{*}{VII }  & $f_1=xz+z^3$                     & \multirow{2}{*}{4} & \multirow{2}{*}{7}  & \multirow{2}{*}{5} & \multirow{2}{*}{6} & \multirow{2}{*}{$a\neq \pm1,0,\frac{5}{4},\frac{1}{2},\frac{3}{2}$} \\
                       & $f_2=y^2z+xz^2+az^4+z^5$        &                    &                     &                    &                    &                                                                     \\[3pt]
 \multirow{2}{*}{VIII } & $f_1=xz+z^4+az^5+bz^7$           & \multirow{2}{*}{3} & \multirow{2}{*}{13} & \multirow{2}{*}{6} & \multirow{2}{*}{8} & \multirow{2}{*}{$a-a^2\neq b$}                                      \\
                       & $f_2=yz^2+z^4+z^5$              &                    &                     &                    &                    &                                                                    
\end{tabular}
}
\caption{Non-simple corank one singularities, $(\CC^3,0)\to(\CC^4,0)$ and of $\Ascr_e\text{-}\codim-\texttt{\#}parameters \leq 4$. None have quadruple points. $^\dagger$ See \cite[Appendix]{Houston1999a}.}
\label{tab:nonsimplec3c4}
\end{table}

From \cref{main}, it is easy to see that none of the non-simple singularities in \cref{tab:nonsimplec3c4} have a good real perturbation. 

\begin{remark}
Observe that \cref{cor:aftermain2} gives another necessary condition to have a good real perturbation. For example, a map germ $f:(\CC^4,0)\to(\CC^5,0)$ that has singular $D^4$ cannot have a good real perturbation, even when $d_4=1$ and $\mu(D^4)=1$.
\end{remark}

\bibliographystyle{myalpha.bst}
\bibliography{FullBib.bib}
\end{document}